\documentclass[12pt]{amsart}
\usepackage{amsfonts, amsmath, amssymb,amsthm}
\usepackage{amsfonts}
\usepackage{enumitem}
\usepackage[a4paper]{geometry}
\usepackage{mathtools}
\usepackage[utf8]{inputenc}

\usepackage{algorithm}
\usepackage{algorithmic}

\usepackage{longtable}
\usepackage{booktabs}
\usepackage{listings}

\usepackage{color}

\newtheorem{theorem}{Theorem}[section]

\newtheorem{question}[theorem]{Question}
\newtheorem{conjecture}[theorem]{Conjecture}

\newtheorem{proposition}[theorem]{Proposition}

\theoremstyle{definition}
\newtheorem{example}[theorem]{Example}

\newtheorem{definition}[theorem]{Definition}

\newtheorem{remark}[theorem]{Remark}
\newtheorem{remarks}[theorem]{Remarks}

\newcommand{\half}{\frac{1}{2}}

\newcommand{\TT}{\mathbb{T}}

\newcommand{\RR}{\mathbb{R}}
\newcommand{\ZZ}{\mathbb{Z}}
\newcommand{\QQ}{\mathbb{Q}}
\newcommand{\Projn}{\Ran / \RR \One}
\newcommand{\Rn}{\RR^{n}}
\newcommand{\Zn}{\ZZ^{n}}
\newcommand{\Qn}{\QQ^{n}}
\newcommand{\Ran}{\RR^{n+1}}

\newcommand{\ta}{\oplus}
\newcommand{\tm}{\odot}

\newcommand{\eqdim}[1]{\mathrm{e}_{\mathrm{tr}}(#1)}

\newcommand{\varx}{{\bf x}}
\newcommand{\vary}{{\bf y}}

\newcommand{\varu}{{\bf u}}
\newcommand{\varv}{{\bf v}}
\newcommand{\varw}{{\bf w}}
\newcommand{\varp}{{\bf p}}
\newcommand{\varq}{{\bf q}}
\newcommand{\varc}{{\bf c}}
\newcommand{\One}{{\bf 1}}
\newcommand{\Zero}{{\bf 0}}

\newcommand{\tnorm}[1]{\parallel #1 \parallel_{\mathrm{tr}} \,}

\newcommand{\Bn}{B_{\mathrm{tr}}^n}

\newcommand{\BRn}{B_{\mathrm{tr},R}^n}

\newcommand{\San}{S_{\mathrm{tr}}^{n-1}}

\newcommand{\tchrnum}[1]{\chi_{\mathrm{tr}}(#1)}

\newcommand{\norm}[1]{\parallel #1 \parallel}
\newcommand{\abs}[1]{\lvert #1 \rvert}

\numberwithin{equation}{section}

\usepackage{pgfplots}
\pgfplotsset{compat = newest}
\usepgfplotslibrary{fillbetween}

\usepackage{calc}
\usepackage{tikz}
\usetikzlibrary{decorations.markings}
\tikzstyle{vertex}=[circle, draw, inner sep=0pt, minimum size=6pt]

\usepackage{graphicx}
\graphicspath{ {./images/} }
\usetikzlibrary{shapes.geometric}
\usetikzlibrary{shapes.geometric}
\usetikzlibrary{arrows}

\begin{document}

\title[]{On the chromatic number and equilateral dimension of $\mathbb{R}^n$ with the tropical norm}
\author{Amnon Rosenmann}
\email[]{rosenmann@math.tugraz.at}
\date{}

\begin{abstract}
We study the tropical chromatic number of $\mathbb{R}^n$, $\chi_{\mathrm{tr}}(\mathbb{R}^n)$, the minimal number of colors needed to color $\RR^n$, so that no two points at tropical unit distance have the same color. It is the tropical analogue of the well-known Hadwiger-Nelson problem in $\mathbb{R}^2$.
We have $\displaystyle \binom{n+1}{\lfloor (n+1)/2 \rfloor} \leq \tchrnum{\Rn} \leq 2^n$ for every $n$, where the lower bound comes from Sperner's antichain bound on a maximal equilateral set, as shown by Swanepoel.  
It is conjectured that $\chi_{\mathrm{tr}}(\mathbb{R}^n) = 2^n$, which is known to be the case for the measurable chromatic number.
By constructing a graph with 62 vertices and 577 edges we demonstrate that 
$\chi_{\mathrm{tr}}(\mathbb{R}^3)=8$.
We also construct a graph in $\ZZ^4$ with 37 vertices and 386 edges that is 11-colorable but not 10-colorable, which is above Sperner's lower bound of 10. 
\end{abstract}
\subjclass[2020]{05C15,14T99, 05C12, 15A80, 52C10, 51K99, 52A20, 05D05, 06A07}
\keywords{Hadwiger-Nelson problem, tropical chromatic number, tropical equilateral dimension, Sperner's antichain theorem}
\maketitle

\section{Introduction}
\label{sec:intro}
The well-known Hadwiger–Nelson problem, or chromatic number of the plane, asks for the minimal number of colors needed to color $\RR^2$ so that no two points at unit distance have the same color.
An upper bound of seven colors was shown by Isbell in 1950 (see \cite{Soi09}, p. 29).
A lower bound of four was given by the Moser brothers in 1961, constructing a graph of seven vertices, known as the Moser spindle \cite{MM61}. A much more complicated graph of 1581 vertices was constructed by de Grey \cite{deG18} in 2018, which requires five colors. A Polymath project was assigned to reduce the number of vertices of a graph requiring five colors and, as far as we know, the current minimal example is a graph of 509 vertices constructed by Parts in 2020 \cite{Paarts20}.
It follows that the chromatic number of the plane is 5,6 or 7.

The chromatic number problem was studied in $\Rn$ with respect to different norms. It is a problem about infinite graphs but by the well-known theorem of de Bruijn and Erd\"os \cite{deBE51} there exists a finite subgraph with the same chromatic number (assuming the axiom of choice).
A lower and upper bounds for the chromatic number of $\Rn$ with the Euclidean norm are $\chi(\Rn) \geq (1.239... + o(1))^n$ (Raigorodskii \cite{Ra00}) and $\chi(\Rn) \leq (3+o(1))^n$ (Larman and Rogers \cite{LR72}), with better lower bounds for some lower dimensions (see \cite{CKR18}). 
Exponential lower and upper bounds are also known for general $l_p$-norms, for $1 < p < \infty$, and the chromatic number is exactly $2^n$ in the $l_{\infty}$-norm.
When the norm is determined by a centrally symmetric convex polytope then an upper bound of $(4+o(1))^n$ was obtained by Kupavskiy \cite{Kup11}.

In this work we study the chromatic number of $\Rn$ with respect to the tropical metric, denoted $\tchrnum{\Rn}$. This metric, an additive version of Hilbert projective metric, was introduced by Cohen, Gaubert and Quadrat \cite{CGQ04} (see also \cite{Pue14}).

The closed unit ball with respect to the tropical metric is convex, centrally symmetric and tiles $\Rn$ face-to-face. For such a type of a norm, Bachoc, Bellitto, Moustrou and P\^echer gave an upper bound of $2^n$ for the chromatic number \cite{BBMP19}.
But their proof refers to parallelohedra $P$ which tile $\Rn$ by disjoint copies of the interior of $P$, ignoring the boundary. In the case of the tropical metric it is not enough that $\Rn$ is tiled by disjoint copies of half open unit balls but we have to make sure that opposite facets are of different colors, so we show in detail how this can be achieved (Section~\ref{sec:chromatic_number_Rn}). 
Asymptotically we have a lower bound $\displaystyle \chi_{\mathrm{tr}}(\mathbb{R}^n) \geq \Omega\!\left(\frac{2^{n}}{\sqrt{n}}\right)$. 
Actually, $\displaystyle \binom{n+1}{\lfloor (n+1)/2 \rfloor} \leq \tchrnum{\Rn} \leq 2^n$ for every $n$, where the lower bound was shown by Swanepoel \cite{Swan07}. We conjecture that $\tchrnum{\Rn} = 2^n$, similar in spirit to a conjecture of Bachoc and Robins (see \cite{BBMP19}).
In fact, Bachoc et al. \cite{BBMP19} show that the tropical measurable chromatic number of $\Rn$ is exactly $2^n$. The term \emph{measurable} refers to the requirement that each color is assigned a measurable set. By the way, the word {\it tropical} is not mentioned in \cite{BBMP19} when using this norm. Note also the minor difference that in \cite{BBMP19} they project $\RR^{n+1}/\RR\One$ to the hyperplane $\One^{\bot}$ whereas we project it to $\Rn$.

In $\RR^2$ we construct the tropical analogue of the Moser spindle, which requires four colors (Section~\ref{sec:chromatic_number_R2}).
In $\RR^3$ the tropical chromatic number $8$ is obtained via a graph with 62 vertices and 577 edges that is 8-colorable but not 7-colorable (Section~\ref{sec:chromatic_number_R3}). The vertices and edges of the graph are presented in Appendix~\ref{sec:Appendix B}.
We also demonstrate in Appendix~\ref{sec:Appendix A} a graph in $\ZZ^4$ with 37 vertices and 386 edges that is 11-colorable but not 10-colorable, that is above the lower bound of $\displaystyle \binom{4+1}{\lfloor (4+1)/2 \rfloor} = 10$. 

A problem related to the chromatic number is the equilateral dimension: the maximal cardinality of a set consisting of points of the same distance from one another. This number is, of course, a lower bound to the chromatic number. The equilateral dimension of $\Rn$ with respect to the Euclidean metric, $\mathrm{e}_{l_2}(\Rn)$, is $n + 1$, which is much smaller than the chromatic number, which grows exponentially with $n$,
whereas in the $l_\infty$-norm we have $\mathrm{e}_{l_\infty}(\Rn) = 2^n$ .
For $1 < p < \infty$ we have $\mathrm{e}_{l_p}(\Rn) \geq n+1$ and $\mathrm{e}_{l_1}(\Rn) \geq 2n$ . Kusner \cite{Guy83} conjectured that these two bounds are equalities.
Alon and Pudl\'ak \cite{AP03} showed that for every odd   integral $p$, $\mathrm{e}_{l_p}(\Rn) \leq c_p n \log n$, for some constant $c_p > 0$. 

Petty \cite{Petty71} showed that the equilateral dimension is bounded above by $2^n$ for $\Rn$ equipped with an arbitrary norm and conjectured that a lower bound is $\min \{4,n+1\}$.
Swanepoel and Villa \cite{SV08}, based on the works of Alon and Milman \cite{AM83} and others, showed that the equilateral dimension of every $n$-dimensional normed space is at least $\exp(c \sqrt{\log n})$, for some $c>0$.

We show here that for every $n$ there is an equilateral set of size $\displaystyle \binom{n+1}{\lfloor (n+1)/2 \rfloor}$ in $\Zn$ with respect to the tropical metric (Section~\ref{sec:equilateral}). We then conjecture that this lower bound is the tropical equilateral dimension of $\Rn$ and we show that this conjecture holds in dimensions 2 and 3 and also when the points belong to the tropical $(n-1)$-sphere of radius $R$ and are of tropical distance $2R$ from one another.
After uploading to the arXiv the first version of this article we were informed that the above results concerning the tropical equilateral dimension already appeared in
Swanepoel \cite{Swan07}
(but not under the name {\it tropical}).

For a background material we review after this introduction the tropical metric and the $n$-dimensional tropical unit ball (Section~\ref{sec:background}).

\section{Background}
\label{sec:background}
\subsection{The tropical metric}
The (min) tropical semifield (or min-plus algebra) ($\TT,\ta, \tm$) consists of the set $\TT = \RR \cup {\infty}$ equipped with the operations of tropical addition $\ta$ and tropical multiplication $\tm$, defined by
\begin{equation}
	a \ta b := \min(a,b) , \qquad a \tm b := a+b.
\end{equation}
The identity element for addition is $\infty$ and the identity element for multiplication is $0$ (see \cite{But10}, \cite{MS15}).
A tropical metric that is an additive version of Hilbert projective metric was introduced in \cite{CGQ04} by Cohen, Gaubert and Quadrat. It is direction-dependent: the tropical distance between two points is invariant to translations but not to rotations.
Let $\TT^{\times} = \TT~\setminus~\{\infty\} = \RR$ and let $\Projn$ be the $n$-dimensional tropical projective torus, where $\One = (1,\ldots,1)$, that is, for all $a \in \RR$, $(x_1,\ldots,x_{n+1}) \sim (x_1+a,\ldots,x_{n+1}+a)$.
For each element of $\Projn$, written in homogeneous coordinates as $(x_1 : \ldots : x_{n+1})$, we can choose as a representative the element $(x_1 - x_{n+1} : \ldots : x_n - x_{n+1}:0)$, and then map it bijectively to the element $(x_1 - x_{n+1}, \ldots, x_n - x_{n+1})$ of $\Rn$.

The tropical distance (see \cite{CGQ04}) in $\Projn$ between $\varx = (x_1 : \ldots : x_{n+1})$ and $\vary = (y_1 : \ldots : y_{n+1})$ is
\begin{equation}
	\label{eq:dist}
	d_{\mathrm{tr}}(\varx, \vary) := \max_{1 \leq i,j \leq n+1} \{x_i - y_i -x_j + y_j\}.
\end{equation}
In $\Rn$ the tropical distance between $\varx=(x_1,\ldots,x_n)$ and $\vary=(y_1,\ldots,y_n)$ is
\begin{align*}
	d_{\mathrm{tr}}(\varx, \vary) &:= \max\{\max_{1 \leq i \leq n} \{\abs{x_i - y_i}\}, \max_{1 \leq i,j \leq n} \{x_i - y_i -x_j + y_j\}\} \\
	&= \max\{\max_{1 \leq i \leq n} \{x_i - y_i\},0\}
	- \min\{ \min_{1 \leq i \leq n} \{ x_i - y_i\}, 0\}. \nonumber
\end{align*}
 
\subsection{The tropical unit ball}
The $n$-dimensional tropical (closed) ball of radius $R$ with center at $\varc \in \Rn$ is
\begin{equation}
	\BRn(\varc) := \{\varx \in {\Rn} : d_{\mathrm{tr}}(\varc, \varx) \leq R \}
\end{equation}
(see e.g. \cite{CJS22}, \cite{Ros26}).
We denote by $\BRn$ the tropical ball of radius $R$ that is centered at the origin and when in addition the radius is 1 then it is denoted by $\Bn$.
The tropical ball is a polytope which is a standard convex set as well as a tropically convex and a tropically geodesic set.
(see Figure~\ref{fig:tr_balls}).
The tropical unit ball that is centered at the origin is centrally symmetric. It consists of the set of points $\varx = (x_1, \ldots, x_n)$ that satisfy these two conditions:
\begin{enumerate}[label=(\roman*)]
	\item $\norm{\varx} \; \leq 1$ (here $\norm{\cdot}$ is the standard Euclidean norm);
	\item $\displaystyle{\max_i \{x_i\} - \min_j \{x_j\} \leq 1}$.
\end{enumerate}
Equivalently, fixing a dummy variable $x_0=0$, then
\begin{equation}
	\label{eq:tr_unit_ball}
	\Bn = \{(x_1, \ldots,  x_n) : \vert x_i-x_j \vert \leq 1, \, 0 \leq i,j \leq n \} \}.
\end{equation}
\noindent
The $2^{n+1}-2$ vertices of $\Bn$ are:
\begin{enumerate}[label=(\roman*)]
	\item $\{(x_1, \ldots, x_n) \in \{0,1\}^n \} \setminus\{\Zero\}$;
	\item $\{(x_1, \ldots, x_n) \in \{0,-1\}^n \} \setminus\{\Zero\}$.
\end{enumerate}
Here $\Zero = (0, \ldots, 0)$. \\
The $n(n+1)$ facets (of dimension $n-1$) of $\Bn$ are:
\begin{enumerate}[label=(\roman*)]
	\item $F_i = \{(x_1, \ldots, x_n) : x_i =  1, 0 \leq x_j \leq 1, j \neq i\}$, for $i=1,\ldots,n$;
	\item $F_{-i}=\{(x_1, \ldots, x_n) : x_i =  -1, -1 \leq x_j \leq 0, j \neq i\}$, for $i=1,\ldots,n$;
	\item $F_{ij}=\{(x_1, \ldots, x_n) : 0 \leq x_i \leq 1, x_j = x_i-1, x_j \leq x_k \leq x_i,  k=1,\ldots,n, k \neq i,j\}$, for $i,j=1,\ldots,n, i \neq j$ (when $n>1$).
\end{enumerate}
We have, $F_{-i}=-F_i$ and $F_{ji}=-F_{ij}$.
The union of the above facets is the tropical unit sphere $\San$.

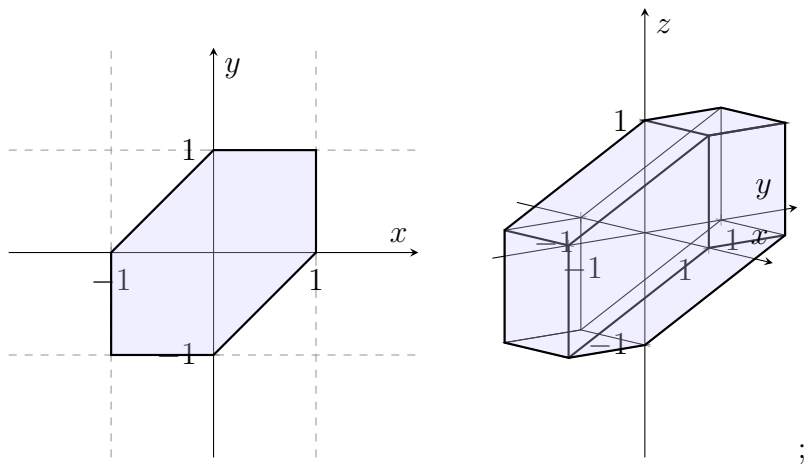
\begin{figure}[h]
	\centering
	\begin{tikzpicture}
		\begin{axis}[
			axis x line=middle,
			axis y line=middle,
			grid = major,
			width=7cm,
			height=7 cm,
			grid style={dashed, gray!80},
			xmin=-2.0,     
			xmax= 2.0,    
			ymin= -2.0,     
			ymax= 2.0,   
			xlabel=$x$,
			ylabel=$y$,
			/pgfplots/xtick={-1.0, 0.0, 1.0}, 
			/pgfplots/ytick={-1.0, 0.0, 1.0}, 
			]
			\draw[thick,black] (1.0,0.0) -- (1.0,1.0) -- (0.0,1.0) -- (-1.0,0.0) -- (-1.0,-1.0) -- (0.0,-1.0) -- cycle [fill=blue!20, opacity=0.3];
			\draw[thick,black] (1.0,0.0) -- (1.0,1.0) -- (0.0,1.0) -- (-1.0,0.0) -- (-1.0,-1.0) -- (0.0,-1.0) -- cycle;
		\end{axis}
	\end{tikzpicture}
	\qquad
	\begin{tikzpicture}
		\begin{axis}[
			view={50}{10},
			axis x line=middle,
			axis y line=middle,
			axis z line=middle,
			grid = major,
			width=9cm,
			height=9cm,
			grid style={dashed, gray !80},
			xmin=-2,     
			xmax= 2,    
			ymin= -2,     
			ymax= 2,   
			zmin= -2,     
			zmax= 2,   
			xlabel=$x$,
			ylabel=$y$,
			zlabel=$z$,
			/pgfplots/xtick={-1.0, 0.0, 1.0}, 
			/pgfplots/ytick={-1.0, 0.0, 1.0}, 
			/pgfplots/ztick={-1.0, 0.0, 1.0}, 
			]
			\coordinate (A1) at (1,0,0);
			\coordinate (A2) at (1,1,0);
			\coordinate (A3) at (0,1,0);
			\coordinate (A4) at (0,0,1);
			\coordinate (A5) at (1,0,1);
			\coordinate (A6) at (1,1,1);
			\coordinate (A7) at (0,1,1);

			\coordinate (B1) at (-1,0,0);
			\coordinate (B2) at (-1,-1,0);
			\coordinate (B3) at (0,-1,0);
			\coordinate (B4) at (0,0,-1);
			\coordinate (B5) at (-1,0,-1);
			\coordinate (B6) at (-1,-1,-1);
			\coordinate (B7) at (0,-1,-1);

			\draw [thick] (A1) -- (A2) -- (A6) -- (A5) -- cycle;
			\draw (A3) -- (A2) -- (A6) -- (A7) -- cycle;
			\draw [thick] (A4) -- (A5) -- (A6) -- (A7) -- cycle;
			\draw (B1) -- (B2) -- (B6) -- (B5) -- cycle;
			\draw [thick] (B3) -- (B2) -- (B6) -- (B7) -- cycle;
			\draw (B4) -- (B5) -- (B6) -- (B7) -- cycle;
			\draw (A7) -- (B1);
			\draw [thick] (A4) -- (B2);
			\draw (A3) -- (B5);
			\draw [thick] (A5) -- (B3);
			\draw (A2) -- (B4);
			\draw [thick] (A1) -- (B7);
			\draw [thick] (B7) -- (B4) -- (A2);
			\draw [thick] (A2) -- (B4) -- (B7) -- (B6) -- (B2) -- (A4) -- (A7) -- (A6) -- cycle [fill=blue!20, opacity=0.3];
		\end{axis}
	\end{tikzpicture};
	\caption{The tropical unit ball in the plane (an hexagon, left) and in space (a rhombic dodecahedron with four-sided faces, right)}
	\label{fig:tr_balls}
\end{figure}

In a tropical $n$-dimensional ball $\BRn$ of radius $R$, two points $\varp$ and $\varq$ are of tropical distance $2R$ from each other, that is, antipodal points, if and only if there is a tropical geodesic joining them that passes through the center of $\BRn$, and this happens if and only if $\varp$ and $\varq$ belong to opposite facets of $\BRn$: $\varp \in F_{i}$ and $\varq \in F_{-i}$ or $\varp \in F_{ij}$ and $\varq \in F_{ji}$. It means that, when $n>1$, given a point $\varp \in \partial \BRn$, there are infinitely many points $\varq \in \partial \BRn$, such that $d_{\mathrm{tr}}(\varp, \varq) = 2R$, the tropical diameter of $\BRn$.

The collection of $n$-dimensional tropical unit balls $\Bn({\varc})$ with centers at the lattice
\begin{equation}
	\Lambda = \{ {\varc} = (c_1,\ldots,c_n) \in\Zn : \sum_{i=1}^{n} c_i \equiv 0 \pmod {n+1} \}
	\label{eq:center} 
\end{equation}
forms a honeycomb of $\Rn$ (see \cite{Ros26}).
The set of vectors $\{\varu_1,\ldots,\varu_n\}$, where $\varu_i = {\bf e_i} + \One$, ${\bf e_i}$ are the standard unit vectors and $\One = (1,\ldots,1) \in \Rn$,
forms a basis for $\Lambda$ over $\ZZ$ in which
\begin{equation}
	\label{eq:Z-basis}
	{\varc} = (c_1,\ldots,c_n) = \sum_{i=1}^{n} \left(c_i-\frac{s}{n+1}\right) \varu_i,
\end{equation}
where $ s=\sum_{i=1}^{n} c_i$.

\section{The tropical chromatic number of $\Rn$}
\label{sec:chromatic_number_Rn}
\begin{definition}
	The {\it tropical chromatic number} of $\Rn$, $\tchrnum{\Rn}$, is the minimal number of colors needed to apply to points of $\Rn$ in such a way that no two points of tropical unit distance from one another have the same color.
\end{definition}
The unit ball with respect to the tropical norm is a parallelohedron and for such a norm Bachoc, Bellitto, Moustrou and P\^echer \cite{BBMP19} gave an upper bound of $2^n$ for the chromatic number. But their proof refers to open parallelohedra $P$ which tile $\Rn$ by disjoint copies of $P$. In the case of the tropical metric it is not enough that $\Rn$ is tiled by disjoint copies of half open unit balls but we have to make sure that opposite facets are assigned different colors.
So, although it is not difficult to satisfy this condition, we show here in detail one way of achieving it.
The colors of the inner tiles are assigned as in \cite{BBMP19}.

Let $\prec$ be the following ``reverse lexicographic order'', a total order on $\Rn$. Given $\varx = (x_1,\ldots,x_n) \neq \vary = (y_1,\ldots,y_n)$ then
$\varx \prec \vary$ if there exists $0 \leq k \leq n-1$, such that $x_n=y_n,
\ldots, x_{n-k+1}=y_{n-k+1}$ and $x_{n-k}<y_{n-k}$. Otherwise, $\vary \prec \varx$.

\begin{theorem}
	\label{thm:$2^n$-colorable}
	$\Rn$ is $2^n$-colorable with respect to the tropical metric.
\end{theorem}
\begin{proof}
	The assertion is clearly true for $n=1$, so assume that $n \geq 2$.
	We construct a honeycomb of $\Rn$ with cells that are tropical balls of radius 1 to form a coloring with respect to tropical distance 2 (it enables us to work over $\ZZ$ instead of $\half \ZZ$ and by rescaling by half at the end we obtain a coloring for tropical distance 1).
	The centers of the cells form a lattice $\Lambda = \{ {\varc} = (c_1,\ldots,c_n) \in\Zn : \sum_{i=1}^{n} c_i \equiv 0 \pmod {n+1} \}$
	with a $\ZZ$-basis $\mathcal{B}$ consisting of the vectors $\varu _i = {\bf e_i}+\One$ of tropical norm 2, as in \eqref{eq:Z-basis}.
	
	Each point of $\Rn$ is given a color indexed by $(i_1,\ldots,i_n) \in \{0,1\}^n$ in the following way.
	The center ${\varc} = \sum_{i=1}^{n} m_i \varu_i$ of the cell $\Bn({\varc})$ is assigned a color representing its element in the quotient group $\Lambda / 2\Lambda$:
	\begin{equation}
		\label{eq:center_color}
		\chi(\varc) = (m_1 (\mathrm{mod}~2), \ldots, m_n (\mathrm{mod}~2)).
	\end{equation}
	The interior of the cell $\Bn(\varc)$ is assigned the same color $\chi({\varc})$.
	Let $\varx$ be  a boundary point of the cells with centers $\varc_i$, for $i=1,\ldots,r$.
	Let $\varw_i=\varx-\varc_i$. Then
	$$
	\chi(\varx) = \chi(\varc_j), \;\mbox{where}\; \varw_j = \min_{\prec} \{\varw_1,\ldots,\varw_r\},
	$$
	that is, the minimum is w.r.t. the total order $\prec$.
	
	Let $\varp, \varp' \in \Rn$ with $d_{\mathrm{tr}}(\varp, \varp') = 2$. We need to show that $\chi(\varp) \neq \chi(\varp')$. Let $\varp \in \Bn(\varc)$ with $\chi(\varp) = \chi(\varc)$ and let $\varp' \in \Bn(\varc')$ with $\chi(\varp') = \chi(\varc')$.
	If $\varc = \varc'$ then $\varp, \varp' \in \partial \Bn(\varc)$ are antipodal points.
	Otherwise, if $\varc \neq \varc'$ then if, in addition, $\chi(\varp) = \chi(\varp')$ then by \eqref{eq:center_color}, $\varc' - \varc = 2\sum_i k_i\varu_i$, for some $k_i \in \ZZ$ and for $\varu_i \in \mathcal{B}$. It follows that $\varc'' = \half(\varc+\varc') = \varc + \sum_i k_i\varu_i$ is the center of the cell $\Bn(\varc'')$ and the (standard) line segment connecting $\varc$ with $\varc'$ (a tropical geodesic, see \cite{Ros26}) passes through $\varc''$.
	This line segment passes through two antipodal points of $\Bn(\varc'')$ of tropical distance 2 from one another, so these points must be $\varp, \varp'$.
	It follows that in order to complete the proof it suffices to show that the antipodal points $\varp$ and $\varp'$ are assigned different colors.
	Since we are not interested in the exact colors but only in showing that $\chi(\varp) \neq \chi(\varp')$, we may assume, w.l.o.g., 
	that $\varp$ and $\varp'$ are boundary points of $\Bn$, the tropical unit ball with center at the origin.
	We look at the two different types of facets of $\Bn$: $F_{\pm i}$ and $F_{ij}$.
	\\ \\
	{\bf Case 1:} $\varp = (p_1,\ldots,p_n) \in F_{-i}$, that is, $-1 \leq p_k \leq 0$, for all $k$, and $p_i = -1$. 
	Let
	\begin{align*}
		S_{1} &= \{ k \in [1..n] \;:\; p_k = -1 \}, \\
		S_{2} &= \{ k \in [1..n] \;:\; p_k = 0 \}, \\
		S_{3} &= \{ k \in [1..n] \;:\; -1 < p_k < 0 \},
	\end{align*}
	where $i \in S_{1}$ and $\vert S_{1} \vert + \vert S_{2} \vert + \vert S_{3} \vert = n$.
	The point $\varp$ is a boundary point of more than one cell, and $\chi(\varp) = \chi(\varc)$, where $\varc = (c_1,\ldots,c_n)$ 
	is the largest, w.r.t. $\prec$, among the centers of these cells.
	One can check that this condition and the fact that $d_{\mathrm{tr}}(\varc, \varp) = 1$ imply that $p_k \leq c_k$, for all $k$.
	It follows that $\sum_{k=1}^{n} c_k = 0$ (for example, $\varc = \Zero$), otherwise, if $\sum_{k=1}^{n} c_k = n+1$, we will have $d_{\mathrm{tr}}(\varc, \varp) > 1$.
	Thus, $c_k=0$, for $k \in S_{3}$, and $c_k = p_k+1$, for $\vert S_{1} \vert$ of the largest (w.r.t. $\prec$) indices in $S_{1} \cup S_{2}$.
	
	Let $\varp' = (p'_1,\ldots,p'_n) \in F_{i}$ be an antipodal point to $\varp$. Let 
	\begin{align*}
		S'_{1} &= \{ k \in [1..n] \;:\; p'_k = 1 \}, \\
		S'_{2} &= \{ k \in [1..n] \;:\; p'_k = 0 \}, \\
		S'_{3} &= \{ k \in [1..n] \;:\; 0 < p'_k < 1 \},
	\end{align*}
	where $i \in S'_{1}$ and $\vert S'_{1} \vert + \vert S'_{2} \vert + \vert S'_{3} \vert = n$.
	Let $\varc' = (c'_1,\ldots,c'_n)$ be the largest, w.r.t. $\prec$, among the centers of the cells for which $\varp'$ is a boundary point,
	and, as above, $p'_k \leq c'_k$, for all $k$.
	Then, we have $c'_k=1$, for $k \in S'_{3}$, and 
	$c_k = p_k+1$, for $\vert S'_{2} \vert + 1$ of the largest (w.r.t. $\prec$) indices in $S'_{1} \cup S'_{2}$, so that $\sum_k c'_k = n+1$.
	
	Let ${\varc} = \sum_{k=1}^{n} m_k \varu_k$ and ${\varc'} = \sum_{k=1}^{n} m'_k \varu_k$.
	Since the sum of the entries of each $\varu_k$ is $n+1$ then $\sum_{k=1}^{n} m_k = 0$ and $\sum_{k=1}^{n} m'_k = 1$,
	which implies by \eqref{eq:center_color} that $\chi(\varp) = \chi(\varc) \neq \chi(\varc') = \chi(\varp')$.
	\\ \\
	{\bf Case 2:} $\varp = (p_1,\ldots,p_n) \in F_{ij}$, where $0 < p_i < 1$, $p_j=p_i-1$, and $p_j \leq p_k \leq p_i$, for $k \neq i,j$. Let also assume, w.l.o.g., that $j>i$.
	We have $\chi(\varp) = \chi(\varc)$, where $\varc = (c_1,\ldots,c_n)$ is the largest, w.r.t. $\prec$, among the centers of the cells that contain $\varp$ as a boundary point. Then $\sum_{k=1}^{n}c_k = 0$.
	One can check that when $\varp \in F_{kl}$ (in addition to $\varp \in F_{ij}$) with respect to a tropical ball with center $\varc''$ then necessarily $p_k=p_i$ and $p_l=p_j$, otherwise one cannot satisfy both $d_{\mathrm{tr}}(\varp, \varc'') = 1$ and $\sum_{i=1}^{n} c''_i \equiv 0 \pmod {n+1}$. 
	Let
	\begin{align*}
		T_{1} &= \{ k \in [1..n] \;:\; p_k = p_i \}, \\
		T_{2} &= \{ k \in [1..n] \;:\; p_k = p_j \}, \\
		T_{3} &= \{ k \in [1..n] \;:\; p_j < p_k < p_i \},
	\end{align*}
	where $\vert T_{1} \vert + \vert T_{2} \vert + \vert T_{3} \vert = n$.
	Then, necessarily, $c_k = 0$, for $k \in T_3$. 
	The other entries of $\varc$, those of index $k \in T_1 \cup T_2$, have values in $\{-1,0,1\}$ in the following way. We arrange these indices by numerical order. We start with the maximal index, say $k$. If $k \in T_2$ then $c_k = 0$. If $k \in T_1$ then $c_k=1$ and $c_l=-1$, where $l$ is the minimal index in $T_2$. We continue in the same way after removing the handled indices from $T_1 \cup T_2$. The process ends when $T_1$ or $T_2$ (or both) are cleared, and then $c_k=0$ for the rest of the indices. 
	At the end, the same number $l \geq 0$ of entries of $\varc$ are of value 1 and of value $-1$.
	
	Let $\varp' = (p'_1,\ldots,p'_n) \in F_{ji}$ be an antipodal point to $\varp$, where $0 < p'_j < 1$, $p'_i=p'_j-1$ and $p'_i \leq p'_k \leq p'_j$, for $k \neq i,j$. Let
	\begin{align*}
		T'_{1} &= \{ k \in [1..n] \;:\; p'_k = p'_j \}, \\
		T'_{2} &= \{ k \in [1..n] \;:\; p'_k = p'_i \}, \\
		T'_{3} &= \{ k \in [1..n] \;:\; p'_i < p'_k < p'_j \},
	\end{align*}
	where $\vert T'_{1} \vert + \vert T'_{2} \vert + \vert T'_{3} \vert = n$.
	Let $\varc' = (c'_1,\ldots,c'_n)$ be the largest, w.r.t. $\prec$, among the centers of the cells for which $\varp'$ is a boundary point.
	Then, as above for $\varc$, $c'_k = 0$, for $k \in T'_3$ and the other entries of $\varc'$, those of index $k \in T'_1 \cup T'_2$ are in $\{-1,0,1\}$ by the same procedure as for $\varc$.
	We get that $\sum_k c'_k = 0$, and since $i \in T'_2$, $j \in T'_1$ and $i<j$ then the process above guarantees that some $l'>0$ entries of $\varc'$ are 1 and $l'$ entries are $-1$.
	
	Since $\varc$ is the sum of, say $m \geq 0$, vectors of the form ${\bf e_k}-{\bf e_l} = \varu_k-\varu_l$ and, similarly, $\varc'$ is the sum of $m' > 0$ such vectors, we get that the coordinates of $\varc$ and $\varc'$ with respect to the basis $\mathcal{B}$ are exactly the entries $c_k$ and $c'_k$. It follows that in order to prove that $\chi(\varc) \neq \chi(\varc')$ we need to show that for some $k \in [1..n]$, $\varc_k$ and $\varc'_k$ are of different parity.
	
	Let us assume that $\chi(\varc) = \chi(\varc')$. We know that at least one entry of $\varc'$ is 1 and at least one is $-1$. Let $k$ be the maximal index for which $c'_k=1$ and $l$ the minimal index for which $c'_l=-1$. Then necessarily $c_k=1$. It cannot be $-1$ since then there would be an index $k'>k$ for which $c_{k'}=1$, which is not possible. Similarly, $c_l=-1$. Continuing with the next maximal index for which the entry of $\varc'$ is 1 and the minimal index with entry $-1$, and so on, at some point we reach the index $j$ or $i$. Suppose that we first reach the index $j$, so that $c'_j=1$. But then, in order for $\varc$ to be the largest center w.r.t. $\prec$ then necessarily $c_j=0$ (remember that $p_j<0$), which is of different parity than that of $c'_j$. Similarly, if we first reach the index $i$ then necessarily $c'_i=-1$ and $c_i=0$, again they are of different parity. Clearly, the same happens if both indices $i$ and $j$ are reached at the same time. It follows that $\chi(\varc) \neq \chi(\varc')$ and thus $\chi(\varp) \neq \chi(\varp')$.
	\\ \\
	{\bf Case 3:} $\varp = (p_1,\ldots,p_n) \in F_{ij} \cap F_{-l}$, that is, $p_i=0$, $p_j=-1$ and $-1 \leq p_k \leq 0$, for $k \neq i,j$.
	In this case, the center $\varc$ for which $\chi(\varp) = \chi(\varc)$ is determined as in case 1. If $\varp' = (p'_1,\ldots,p'_n) \in F_{l}$ is an antipodal point to $\varp$ then we are in case 1 and $\chi(\varp) \neq \chi(\varp')$. Otherwise, suppose that $\varp' = (p'_1,\ldots,p'_n) \in F_{ji}$ is an antipodal point to $\varp$, where $0 \leq p'_j < 1$, $p'_i=p'_j-1$ and $p'_i \leq p'_k \leq p'_j$, for $k \neq i,j$. If $i<j$ then, as in case 2, we get that $(c_j,c_j') = (0,1)$ or $(c_i,c_i') = (0,-1)$. If $i>j$ then we get that $(c_i,c_i') = (1,0)$ or $(c_j,c_j') = (-1,0)$. In both cases $\chi(\varc) \neq \chi(\varc')$ and thus $\chi(\varp) \neq \chi(\varp')$. 
\end{proof}

\begin{proposition}
	The tropical chromatic number of $\Rn$ satisfies 
	$\displaystyle \chi_{\mathrm{tr}}(\Rn) \geq \Omega\!\left(\frac{2^{n}}{\sqrt{n}}\right).$
\end{proposition}
\begin{proof}
	In Proposition~\ref{prop:lower_edim} we show that the maximal number of points in $\Rn$ that form a tropical equilateral set is at least $\displaystyle \binom{n+1}{\lfloor (n+1)/2 \rfloor}$.
	Asymptotically, we get by Stirling formula that 
	$$\displaystyle \binom{n+1}{\lfloor (n+1)/2 \rfloor} \sim
	\frac{\sqrt{2}\cdot 2^{n+1}}{\sqrt{\pi (n+1)}}$$
	and hence
	$$
	\chi_{\mathrm{tr}}(\Rn) \geq \Omega\!\left(\frac{2^{n}}{\sqrt{n}}\right).
	$$
\end{proof}

\begin{conjecture}
	\label{conj:tr_chi_Rn}
	The tropical chromatic number of $\Zn$ and of $\Rn$ is $2^n$.
\end{conjecture}
\begin{remarks}
	\begin{enumerate}
		\item The conjecture is verified in dimensions $n \leq 3$ (see Sections~\ref{sec:chromatic_number_R2}, \ref{sec:chromatic_number_R3}). 
		\item The number $m_1(\Rn, \norm{\cdot})$, introduced by Larman and Rogers \cite{LR72}, is the supremum, over measurable unbounded sets $S \subset \Rn$ containing no two points at unit distance (with respect to the norm $\norm{\cdot}$), of the following limit of densities
		$$
		\limsup_{R \to \infty} \frac{\lvert S \cap [-R,R]^n  \rvert}{(2R)^n}.
		$$
		It follows that $\chi_m(\Rn, \norm{\cdot}) \geq \frac{1}{m_1(\Rn, \norm{\cdot})}$, where $\chi_m$ refers to the chromatic number with the restriction that the subset of each color is measurable. It is conjectured by Bachoc and Robins (see \cite{BBMP19}) that $m_1(\Rn, \norm{\cdot}) = 2^{-n}$ when the unit ball w.r.t. the norm $\norm{\cdot}$ is a parallelohedron (as is the tropical ball). Specifically, Bachoc et al. \cite{BBMP19} proved that this conjecture holds for $n = 2$ and for the norm induced by the parallelohedron that is the Voronoi regions of the lattices $A_n$, for $n \geq 2$. These Voronoi regions are exactly the tropical balls (although not mentioned in this name in \cite{BBMP19}). The proof uses a discretization of the problem by bounding $m_1(\Rn, \tnorm{\cdot})$ from above by the supremum of densities of independent sets of a graph with edges connecting vertices of unit tropical distance.
	\end{enumerate}
\end{remarks}
\begin{example}
	Appendix~\ref{sec:Appendix A} is about a graph in $\ZZ^4$ with 37 vertices and 386 edges, connecting vertices of tropical distance 2. It is 11-colorable but not 10-colorable, as verified by different solvers.
	That is above the lower bound $\displaystyle \binom{4+1}{\lfloor (4+1)/2 \rfloor} = 10$.
\end{example}

\section{The tropical chromatic number of the plane}
\label{sec:chromatic_number_R2}
As was shown in Bachoc et al. \cite{BBMP19}, the chromatic number of $\RR^2$ is 4 when the unit ball is a parallelohedron. This applies also to the tropical ball.
As for a lower bound to $\chi_{\mathrm{tr}}(\RR^2)$, we give in the following example a
tropical analogue of the Moser spindle \cite{MM61} which requires four colors.
\begin{example}[The tropical Moser spindle]
	Let $G(V,E)$ be the following plane graph with 7 vertices and 11 edges of tropical unit length. \\
	$V$ consists of the vertices:
	$
	v_1=(1,1), v_2=(0,0), v_3=(1,0), v_4=(0.5,0), v_5=(1.5,0.5), v_6=(0,-1), v_7=(1,-0.5).
	$
	\\
	$E$ consists of the edges:
	$
	(v_1,v_2), (v_1,v_3), (v_2,v_3), (v_1,v_4), (v_1,v_5), (v_4,v_5), (v_2,v_6), \\
	(v_3,v_6), (v_4,v_7), (v_5,v_7), (v_6,v_7).
	$
	A 4-coloring of the graph is shown in Figure~\ref{fig:tr_Moser_spindle}.
\end{example}
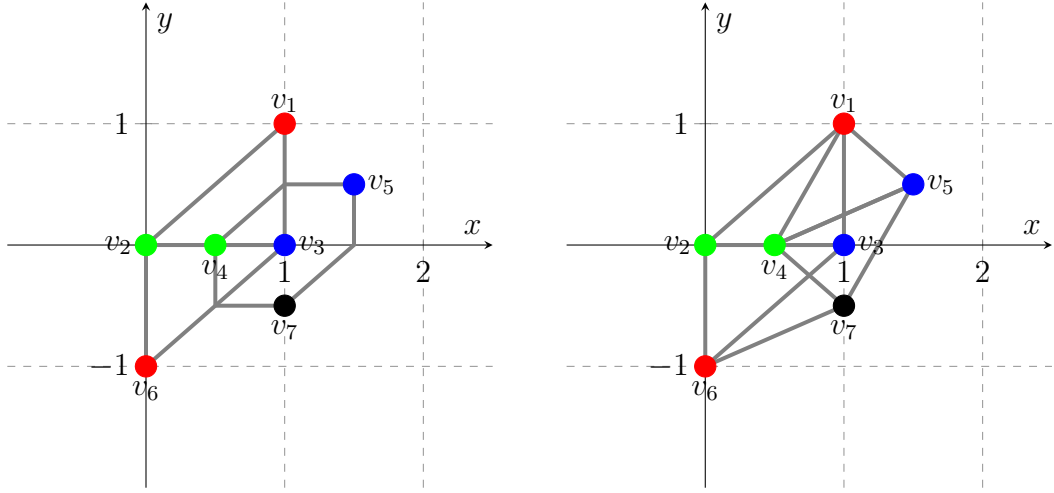
\begin{figure}[h]
	\centering
	\begin{tikzpicture}
	\begin{axis}[
		axis x line=middle,
		axis y line=middle,
		grid = major,
		width=8cm,
		height=8cm,
		grid style={dashed, gray!80},
		xmin=-1.0,     
		xmax= 2.5,    
		ymin= -2.0,     
		ymax= 2.0,   
		xlabel=$x$,
		ylabel=$y$,
		/pgfplots/xtick={0.0,1.0,2.0}, 
		/pgfplots/ytick={-1.0,0.0,1.0}, 
		]
		\coordinate (v1) at (1,1);
		\coordinate (v2) at (0,0);
		\coordinate (v3) at (1,0);
		\coordinate (v4) at (0.5,0);
		\coordinate (v5) at (1.5,0.5);
		\coordinate (v6) at (0,-1);
		\coordinate (v7) at (1,-0.5);
		\draw [line width=1.5,gray] (v1) -- (v2) -- (v3) -- cycle;
		\draw [line width=1.5,gray] (v2) -- (v6) -- (v3) -- cycle;
		\draw [line width=1.5,gray] (1,0.5) -- (0.5,0) -- (0.5,-0.5) -- (v7) -- (1.5,0) -- (v5) -- cycle;
		\node at (1,1.17){$v_1$};
		\node at (-0.2,0){$v_2$};
		\node at (1.2,0){$v_3$};
		\node at (0.5,-0.2){$v_4$};
		\node at (1.7,0.5){$v_5$};
		\node at (0,-1.2){$v_6$};
		\node at (1,-0.7){$v_7$};
		\filldraw [red] (1,1) circle (4pt);
		\filldraw [green] (0,0) circle (4pt);
		\filldraw [blue] (1,0) circle (4pt);
		\filldraw [green] (0.5,0) circle (4pt);
		\filldraw [blue] (1.5,0.5) circle (4pt);
		\filldraw [red] (0,-1) circle (4pt);
		\filldraw [black] (1,-0.5) circle (4pt);
	\end{axis}
	\end{tikzpicture}
	\qquad
	\begin{tikzpicture}
		\begin{axis}[
			axis x line=middle,
			axis y line=middle,
			grid = major,
			width=8cm,
			height=8cm,
			grid style={dashed, gray!80},
			xmin=-1.0,     
			xmax= 2.5,    
			ymin= -2.0,     
			ymax= 2.0,   
			xlabel=$x$,
			ylabel=$y$,
			/pgfplots/xtick={0.0,1.0,2.0}, 
			/pgfplots/ytick={-1.0,0.0,1.0}, 
			]
			\coordinate (v1) at (1,1);
			\coordinate (v2) at (0,0);
			\coordinate (v3) at (1,0);
			\coordinate (v4) at (0.5,0);
			\coordinate (v5) at (1.5,0.5);
			\coordinate (v6) at (0,-1);
			\coordinate (v7) at (1,-0.5);
			\draw [line width=1.5,gray] (v1) -- (v2) -- (v3) -- cycle;
			\draw [line width=1.5,gray] (v2) -- (v6) -- (v3) -- cycle;
			\draw [line width=1.5,gray] (v1) -- (v4) -- (v5) -- cycle;
			\draw [line width=1.5,gray] (v4) -- (v7) -- (v5) -- cycle;
			\draw [line width=1.5,gray] (v6) -- (v7);
			\node at (1,1.17){$v_1$};
			\node at (-0.2,0){$v_2$};
			\node at (1.2,0){$v_3$};
			\node at (0.5,-0.2){$v_4$};
			\node at (1.7,0.5){$v_5$};
			\node at (0,-1.2){$v_6$};
			\node at (1,-0.7){$v_7$};
			\filldraw [red] (1,1) circle (4pt);
			\filldraw [green] (0,0) circle (4pt);
			\filldraw [blue] (1,0) circle (4pt);
			\filldraw [green] (0.5,0) circle (4pt);
			\filldraw [blue] (1.5,0.5) circle (4pt);
			\filldraw [red] (0,-1) circle (4pt);
			\filldraw [black] (1,-0.5) circle (4pt);
		\end{axis}
	\end{tikzpicture}
	\caption{The tropical Moser spindle. Left: edges as tropical lines, right: edges as straight lines}
	\label{fig:tr_Moser_spindle}
\end{figure}

A 4-coloring of the plane according to the same scheme as in Theorem~\ref{thm:$2^n$-colorable} assigns the points $\{\varv=(v_1,v_2) \in \partial B^2_{\mathrm{tr}}(\varc) : v_2 < c_2 \}$ in the lower part of the boundary of the cell $B^2_{\mathrm{tr}}(\varc)$ with center $\varc = (c_1,c_2)$ the same color as the color of the inner of the cell (see Figure~\ref{fig:4 coloring}).

\begin{figure}[h]
	\centering
\begin{tikzpicture}[scale=0.55, line cap=round, line join=round]
	\def\s{1.00}          
	\def\thinlw{0.5pt}    
	\def\thicklw{2.2pt}   
	\def\vrad{0.18}       
	\def\gap{0.14}        
	
	\definecolor{c0}{HTML}{4D4DB2} 
	\definecolor{c1}{HTML}{E30022} 
	\definecolor{c2}{HTML}{808000} 
	\definecolor{c3}{HTML}{FF7F00} 
	
	\newcommand{\HexVerts}{%
		\coordinate (v0) at (\s,0);
		\coordinate (v1) at (\s,\s);
		\coordinate (v2) at (0,\s);
		\coordinate (v3) at (-\s,0);
		\coordinate (v4) at (-\s,-\s);
		\coordinate (v5) at (0,-\s);
	}
	
	\newcommand{\OpenEdge}[2]{%
		\path (#1) -- (#2)
		coordinate[pos=\gap] (Pstart)
		coordinate[pos=1-\gap] (Pend);
		\draw[line width=\thicklw] (Pstart) -- (Pend);
	}
	
	\newcommand{\HexFill}[3]{
		\begin{scope}[shift={({#1},{#2})}]
			\HexVerts
			\fill[#3!70] (v0)--(v1)--(v2)--(v3)--(v4)--(v5)--cycle;
		\end{scope}
	}
	
	\newcommand{\HexThin}[2]{
		\begin{scope}[shift={({#1},{#2})}]
			\HexVerts
			\draw[black!65, line width=\thinlw]
			(v0)--(v1)--(v2)--(v3)--(v4)--(v5)--cycle;
		\end{scope}
	}
	
	\newcommand{\HexEdgesAndVerts}[3]{
		\begin{scope}[shift={({#1},{#2})}]
			\HexVerts
			{\color{#3!95}
				\OpenEdge{v3}{v4}
				\OpenEdge{v4}{v5}
				\OpenEdge{v5}{v0}
			}
			\fill[#3!95] (v4) circle[radius=\vrad];
			\fill[#3!95] (v5) circle[radius=\vrad];
		\end{scope}
	}
	
	\foreach \m in {-4,...,4}{
		\foreach \n in {-4,...,4}{
			\def\CX{2*\m - \n}
			\def\CY{\m + \n}
			\ifodd\m
			\ifodd\n \def\COL{c3}\else\def\COL{c2}\fi
			\else
			\ifodd\n \def\COL{c1}\else\def\COL{c0}\fi
			\fi
			\HexFill{\CX}{\CY}{\COL}
		}
	}
	\foreach \m in {-4,...,4}{
		\foreach \n in {-4,...,4}{
			\def\CX{2*\m - \n}
			\def\CY{\m + \n}
			\HexThin{\CX}{\CY}
		}
	}
	\foreach \m in {-4,...,4}{
		\foreach \n in {-4,...,4}{
			\def\CX{2*\m - \n}
			\def\CY{\m + \n}
			\ifodd\m
			\ifodd\n \def\COL{c3}\else\def\COL{c2}\fi
			\else
			\ifodd\n \def\COL{c1}\else\def\COL{c0}\fi
			\fi
			\HexEdgesAndVerts{\CX}{\CY}{\COL}
		}
	}
	
\end{tikzpicture}
	\caption{Tropical 4-coloring of the plane}
	\label{fig:4 coloring}
\end{figure}
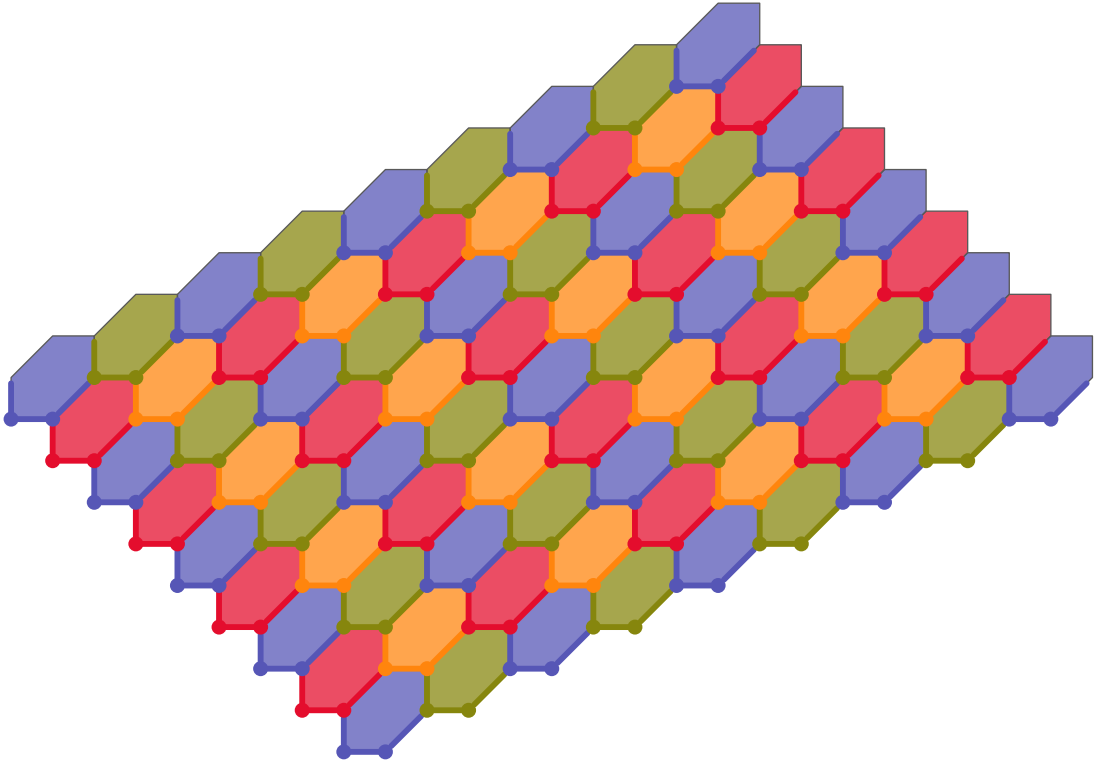

The next example, although not optimal, shows a 6-coloring of $\RR^2$ which does not rely on the tiling of the plane.
\begin{example}
	The plane can be colored with the colors $0, \ldots,5$ in the following way.
	For $\varx=(x_1,x_2)$, let
	$$
	S(\varx)=x_1+x_2, \quad D(\varx)=x_1-x_2.
	$$
	Let the color of $\varx$ be
	\begin{equation}
		\chi(\varx) := (\lfloor S(\varx) \rfloor + 3 \cdot \lfloor D(\varx) \rfloor) \: (\mathrm{mod}~6),
	\end{equation}
	where $\lfloor x \rfloor \in \ZZ$ is the floor function.
	We need to show that given a tropical unit circle $S^1_{\mathrm{tr}}({\varc})$ with center at $\varc$, for each $\varq \in S^1_{\mathrm{tr}}({\varc})$ we have $\chi({\varq}) \neq \chi({\varc})$.
	In table~\ref{tab:tropical 6 colors} we list the sets of possible differences
	$$
	\Delta(\varc,\varq) = \chi({\varq}) - \chi({\varc})
	$$
	for three of the closed line segments $\varc + [(a,b), (c,d)] \subset S^1_{\mathrm{tr}}({\varc})$ to which $\varq$ belongs.
	By symmetry, for $\varq' =  2\varc-\varq$, the antipodal point of $\varq$, we have $\Delta(\varc,\varq') = -\Delta(\varc,\varq')$.
	\begin{center}
		\begin{table}[h]
			\begin{tabular}{cc} 
				\toprule
				Line segment of $\varq$ & $\Delta(\varc,\varq)$ \\
				\midrule
				$\varc + [(0,-1),(1,0)]$ & $\{-1,0,1\} +3 \cdot \{1\} = \{2,3,4\}$ \\
				\hline
				$\varc + [(1,0),(1,1)]$ & $\{1,2\} +3 \cdot \{0,1\} = \{1,2,4,5\}$ \\ 
				\hline
				$\varc + [(0,1),(1,1)]$ & $\{1,2\} +3 \cdot \{-1,0\} = \{-2,-1,1,2\}$ 
			\end{tabular}	
			\vspace{5pt}
			\caption{Differences in color for points that are at tropical distance 1}
			\label{tab:tropical 6 colors}
		\end{table}
	\end{center}
	As can be seen, all the possible differences are not congruent to 0 modulo 6, thus $\chi({\varq}) \neq \chi({\varc})$ for every point $\varq \in  S^1_{\mathrm{tr}}({\varc})$.
\end{example}

\section{The tropical chromatic number of $\RR^3$}
\label{sec:chromatic_number_R3}
We show that Conjecture~\ref{conj:tr_chi_Rn} holds for $n=3$.
\begin{theorem}
	\label{thm:8-colorable}
	The tropical chromatic number of $\RR^3$ is 8.
\end{theorem}
\begin{proof}
	By Theorem~\ref{thm:$2^n$-colorable}, we know that an upper bound to the tropical chromatic number of $\RR^3$ is $2^3=8$. This number is also a lower bound as demonstrated by a graph with vertices in the set $\{0,1,2,3,4\}^3$ that is 8-colorable but not 7-colorable. The graph has 62 vertices and 577 edges, connecting vertices of tropical distance 2 from one another, as seen in Appendix~\ref{sec:Appendix B}.
	The graph was constructed by a process of pruning the cube $\{0,1,2,3,4\}^3$ as long as it was not 7-colorable. We do not claim that it is a minimal graph with respect to the number of vertices and edges that requires 8 colors.
	It contains the following clique of 6 vertices that are of equilateral tropical distance 2:
	$$
	31\,(2,2,2),\;34\,(2,3,1),\;44\,(3,2,1),\;52\,(3,4,3),\;58\,(4,3,3),\;60\,(4,4,2).
	$$

	We converted the graph coloring problem into a Boolean satisfiability problem (SAT) in the form of a DIMACS CNF file, both for 7-coloring and for 8-coloring.
	The CNF file was generated using the standard encoding: one Boolean variable $x(v,c)$ per vertex-color pair, with clauses for at least one color for each vertex $v$:
	$x(v,0) \vee \cdots \vee x(v,k~-~1)$ (for $k=7$ or $k=8$) and at most one color for each vertex: 
	$\neg x(v,c_i) \vee \neg x(v,c_j)$ for all $c_i < c_j$; and a clause
	$\neg x(u,c) \vee \neg x(v,c)$ for every edge $(u,v)$ and color $c$.
	We then ran the SAT solver CaDiCaL \cite{Cadical24} on the CNF files. The output was UNSATISFIABLE for 7-coloring and SATISFIABLE for 8 coloring.
	As an independent check, we formulated the $7$-coloring problem as a binary
	mixed-integer linear feasibility problem and solved it with the exact MILP solver
	HiGHS \cite{HH18}.
	The solver returned the model as infeasible, that is, the graph is not $7$-colorable.
\end{proof}

\section{The tropical equilateral dimension of $\Rn$}
\label{sec:equilateral}
As stated in the Introduction, the following results: Proposition~\ref{prop:lower_edim}, Proposition~\ref{prop:edim_2_3}, Theorem~\ref{thm:edim_ball}, as well as Conjecture~\ref{conj:edim}, already appeared in Swanepoel \cite{Swan07}, which we were aware of only after uploading the first version of this article to the arXiv.

Related to the tropical chromatic number is the equilateral dimension (see the expository paper \cite{Swan04} and the survey papers \cite{MSW01} and \cite{Swan18} about this problem in finite dimensional normed spaces).

\begin{question}
	What is the tropical equilateral dimension of $\ZZ^n$ ($\Rn$), that is, what is the maximal number of points in $\ZZ^n$ ($\Rn$)  that are of the same tropical distance from each other?
\end{question}
Before providing a lower bound to this problem, we show that it suffices to work over $\ZZ$ and not over $\RR$. Let us denote the tropical equilateral dimensions by $\eqdim{\cdot}$.
\begin{proposition}
	$\eqdim{\Rn} = \eqdim{\Qn} = \eqdim{\Zn}$.
\end{proposition}
\begin{proof}
	Clearly, $\eqdim{\Rn} \geq \eqdim{\Qn} \geq \eqdim{\Zn}$. We will show that any clique of size $m$ and of tropical distance $r$ in $\Rn$ implies the existence of a clique of size $m$ and of tropical distance $1$ in $\Qn$ and a clique of size $m$ and of tropical distance $d$ in $\Zn$, for some positive $d \in \ZZ$.

	To see that, suppose that $\eqdim{\Rn} = m$ and $\{\varx^{(1)}, \ldots, \varx^{(m)}\}$ forms an equilateral set in $\Rn$ of tropical distance $r \in \RR$. By rescaling, we may assume that $r=1$. For each pair of points $\varx^{(i)}, \varx^{(j)}$,
	$$
		d_{\mathrm{tr}}(\varx^{(i)}, \varx^{(j)}) = \max_{0 \leq k,l \leq n} \{\varx^{(i)}_k - \varx^{(j)}_k -\varx^{(i)}_l + \varx^{(j)}_l\} = 1,
	$$
	where we added a fixed dummy entry $\varx^{(i)}_0 = \varx^{(j)}_0 = 0$.
	\\
	Let the indices $0 \leq p(i,j), p(j,i) \leq n$ be such that
	$$
		\varx^{(i)}_{p(i,j)} - \varx^{(j)}_{p(i,j)}
		= \max_{0 \leq k \leq n} \varx^{(i)}_k - \varx^{(j)}_k,
	$$
	and
	$$
		\varx^{(j)}_{p(j,i)} - \varx^{(i)}_{p(j,i)}
		= \max_{0 \leq l \leq n} \varx^{(j)}_l - \varx^{(i)}_l.
	$$
	It follows that
	$$
		\varx^{(i)}_{p(i,j)} - \varx^{(j)}_{p(i,j)} + \varx^{(j)}_{p(j,i)} - \varx^{(i)}_{p(j,i)} = 1.
	$$
	
	Let now $y^{(i)}_0, \dots, y^{(i)}_{n}$, with $y^{(i)}_{0} := 0$, for $i = 1, \ldots, m$, be variables.
	Consider the following system of equations 
	$$
		y^{(i)}_{p(i,j)} - y^{(j)}_{p(i,j)} + y^{(j)}_{p(j,i)} - y^{(i)}_{p(j,i)} = 1,
	$$
	for $1 \leq i < j \leq m$, together with the inequalities 
	$$
		y^{(i)}_{p(i,j)} - y^{(j)}_{p(i,j)} \;\geq \;
		y^{(i)}_l - y^{(j)}_l \; \geq\; 
		y^{(i)}_{p(j,i)} - y^{(j)}_{p(j,i)},
	$$
	for $l = 0, \ldots, n$ and $1 \leq i <j \leq m$.
	This is a finite system of linear equations and inequalities over $\ZZ$ that has a real solution.
	Hence, the solution set is a nonempty rational polyhedron and therefore it contains a rational point. Thus, we have an equilateral set of size $m$ and tropical distance 1 in $(\frac{1}{d}\QQ)^n$, where $d$ is the common denominator of the values in the rational solution. Multiplying by $d$ gives an equilateral set of tropical distance $d$ and of size $m$ in $\ZZ^n$.
\end{proof}
The next proposition shows that already an equilateral set with respect to the tropical norm can be of exponential size. 
\begin{proposition}[Swanepoel \cite{Swan07}]
	\label{prop:lower_edim}
	A lower bound to the tropical equilateral dimension of 
	$\Rn$ is $\displaystyle \binom{n+1}{\lfloor (n+1)/2 \rfloor}.$
\end{proposition}
\begin{proof}
	Let $k = \lfloor (n+1)/2 \rfloor$ and let $N = \binom{n+1}{k}$.
	Consider the set $S$ of $N$ points in $\ZZ^{n+1}/\ZZ\One$, written in homogeneous coordinates, with exactly $k$ of the entries being $1$ and $n+1-k$ entries being $0$.
	For each pair of points $\varx, \vary \in S$ it holds
	$$
		d_{\mathrm{tr}}(\varx, \vary) = 
		\max_{1 \leq i,j \leq n+1} \{x_i - y_i -x_j + y_j\} =
		(1-0) - (0-1) = 2.
	$$
	So, we have $N$ different points in $\ZZ^{n+1}/\ZZ\One$ with tropical distance $2$ between every two points.
	Each such point $(x_1 : \ldots : x_n : x_{n+1}) = (x_1 - x_{n+1} : \ldots : x_n - x_{n+1}:0)$ is mapped to $(x_1 - x_{n+1}, \ldots, x_n - x_{n+1}) \in \Zn$, where 
	the  tropical distances between the points are preserved. The result is an equilateral set of tropical distance $2$ consisting of $N$ points in $\{0,1\}^n$ or in $\{-1,0\}^n$.
	We note that these points are vertices of the tropical unit ball.
\end{proof}
\begin{example}
Let $n=4$. Then $k = \lfloor 5/2 \rfloor = 2$ and $N = \binom{5}{2} = 10$. The $10$ points of tropical distance $2$ from each other and written in homogeneous coordinates are
\begin{align*}
	&(1:1:0:0:0),\; (1:0:1:0:0),\; (1:0:0:1:0),\; (1:0:0:0:1),\\ 
	&(0:1:1:0:0),\; (0:1:0:1:0),\; (0:1:0:0:1),\; (0:0:1:1:0),\\
	&(0:0:1:0:1),\; (0:0:0:1:1).
\end{align*}
These points are mapped to
\begin{align*}
	&(1,1,0,0),\; (1,0,1,0),\; (1,0,0,1),\; (0,-1,-1,-1),\; (0,1,1,0), (0,1,0,1),\\
	&(-1,0,-1,-1),\; (0,0,1,1),\; (-1,-1,0-1),\; (-1,-1,-1,0)
\end{align*}
in $\ZZ^4$ to form an equilateral set of tropical distance 2.
\end{example}

Next, we show that in dimensions 2 and 3 the lower bound of Proposition~\ref{prop:lower_edim} is an equality.
\begin{proposition}[Swanepoel \cite{Swan07}]
	\label{prop:edim_2_3}
	$\eqdim{\RR^2} = \binom{3}{1}=3$ and $\eqdim{\RR^3} = \binom{4}{2}=6$.
\end{proposition}
\begin{proof}
	It is easy to see that $\eqdim{\RR^2} < 4$. Let one point $\varp_0$ be at the center of the tropical unit ball (Figure~\ref{fig:tr_balls}) and three points $\varp_1, \varp_2$ and $\varp_3$ on the boundary, set in clockwise direction.
	These 3 points are on 3 adjacent edges, each edge of tropical length 1. But the boundary consists of 6 edges, so the tropical distance between $\varp_1$ and $\varp_3$ is greater than 1.
	In fact, it is known for a general Minkowski space (finite dimensional Banach space) $X$ of dimension 2 that the equilateral dimension of $X$ is either 4 (when the unit ball of $X$ is a parallelogram) or 3 (otherwise) (see, e.g. \cite{Swan04}).
	 
	Next, we show that $\eqdim{\RR^3} = 6$.
	Given $\varx=(x_1,x_2,x_3) \in \RR^3$, its tropical norm is
	$$
		\tnorm{\varx} = \max \{\lvert x_1 \rvert, \lvert x_2 \rvert, \lvert x_3 \rvert, \lvert x_1-x_2 \rvert, \lvert x_1-x_3 \rvert, \lvert x_2-x_3 \rvert\}.
	$$
	Let $T : \RR^3 \to \RR^3$ be the linear isomorphism
	$$
		T(x_1,x_2,x_3) = (x_1 + x_2, x_1 + x_3, x_2 + x_3).
	$$	
	Then
	$$
		\tnorm{T(x_1,x_2,x_3)} = \max \{\lvert x_1 \pm x_2 \rvert, \lvert x_1 \pm x_3 \rvert, \lvert x_2 \pm x_3 \rvert\}.
	$$
	It follows that $\RR^3$ with the tropical norm is linearly isometric to $\RR^3$ with the norm $\parallel \cdot \parallel_{rd}$ defined by
	$$
		\parallel (x_1,x_2,x_3) \parallel_{rd} \; = \max \{\lvert x_1 \pm x_2 \rvert, \lvert x_1 \pm x_3 \rvert, \lvert x_2 \pm x_3 \rvert\},
	$$
	whose unit ball is a rhombic dodecahedron. But Sch\"urmann and Swanepoel
	\cite{SchS06} showed that the equilateral dimension of $\RR^3$ with respect to this norm is 6 and since the equilateral dimension is preserved by linear isometries then $\eqdim{\RR^3} = 6$. 
\end{proof}

\begin{conjecture}[Swanepoel \cite{Swan07}]
	\label{conj:edim}
	The tropical equilateral dimension of $\Rn$ is 
	$$\displaystyle \eqdim{\Rn} = \binom{n+1}{\lfloor (n+1)/2 \rfloor}.$$
\end{conjecture}
\begin{remarks}
	\begin{enumerate}
		\item The binomial coefficient in Conjecture \ref{conj:edim} is the one that appears in the antichain theorem of Sperner \cite{Sperner28} and the set-pair theorem of Bollob\'as \cite{Bol65}.
		\item By Petty \cite{Petty71} the points of a maximal equilateral set $S$ are extreme points of the convex hull $H$ of $S$ and the line through each pair of points of $S$ is orthogonal to two parallel hyperplanes through these points that are also supporting hyperplanes for $H$.
		\item Petty \cite{Petty71} and independently Soltan \cite{Soltan75}, based on Danzer and Gr\"unbaum \cite{DG62}, showed that the equilateral dimension is at most $2^n$. The same upper bound was obtained by F\"uredi, Lagarias and Morgan \cite{FLM91}). This upper bound is achieved if and only if the convex hull of $S$ is a ball of the normed space in the form of a parallelotope with the set $S$ being its set of vertices. Since, for $n>1$, the tropical ball $\Bn$ is not a parallelotope then $$\eqdim{\Rn} < 2^n.$$
		\item The conjecture is supported by some experiments conducted on $\ZZ^4$ and $\ZZ^5$ with several low tropical distances $d$, where $d$ should be even in order to achieve the bound. 
	\end{enumerate}
\end{remarks}

We have seen in the proof of Proposition~\ref{prop:lower_edim} that the lower bound is achieved with an equilateral set of tropical distance 2 consisting of points on the boundary of the tropical unit ball. Next we show that under these conditions this is also an upper bound. 
\begin{theorem}[Swanepoel \cite{Swan07}]
	\label{thm:edim_ball}
	Let $S$ be an equilateral set of tropical distance $2R$ of maximum size that is contained in a tropical sphere of radius $R$. Then
	$$\vert S \vert = \binom{n+1}{\lfloor (n+1)/2 \rfloor}.$$
\end{theorem}
\begin{proof}
	By translating and rescaling we may assume, w.l.o.g., that $S \subset \partial \Bn$, the boundary of the unit ball.
	Each point $\varp \in S$ belongs to a face $f$ of minimal dimension, which is the intersection of all the facets to which $\varp$ belongs.
	Each of the other points belongs to at least one facet that is antipodal to a facet to which $\varp$ belongs. It follows that if $\varp$ is not a vertex of $\Bn$ then by moving it to a vertex of $f$ not only its tropical distance from all the other points is preserved but potentially we can add more points to the equilateral set (in case $S$ were not maximal) since now $\varp$ is contained in more facets.
	It follows that we may assume that a maximal equilateral set $S$ consists of vertices of $\Bn$.
	
	Let us now write each of these vertex points in homogeneous coordinates and, if necessary, add constants so that each point $\varp_i$ has $0 < k_i \leq n+1$ ones and $n+1-k_i$ zeros. But if $I \subseteq [n+1]$ is the set of indices of ones in $\varp_i$ and $J$ is such a set in $\varp_j$ then $d_{\mathrm{tr}}(\varp_i, \varp_j)=2$ if and only if $I \nsubseteq J$ and $J \nsubseteq I$. By Sperner's theorem \cite{Sperner28}, the size of this antichain of indices $I$ is at most $\displaystyle \binom{n+1}{\lfloor (n+1)/2 \rfloor}$,
	and we saw in the proof of Proposition~\ref{prop:lower_edim} that this upper bound is indeed achieved. 
\end{proof}
\begin{remark}
	In general, when an equilateral set of tropical distance $d$ belongs to a tropical sphere then the radius of the sphere is $\frac{d}{2} \leq R \leq d$. For example, the equilateral set $\{(0,0), (1,0), (1,1)\}$ of tropical distance 1 is on the boundary of a tropical ball of radius $\frac{2}{3}$ (and not $\half$) with center at $(\frac{2}{3}, \frac{1}{3})$ (the centroid, in this case).
\end{remark}
A related question is the following.
\begin{question}
	Does every (maximal) equilateral set lie on a tropical sphere?
\end{question}
\begin{remark}
	In $\Rn$ with the Euclidean metric the answer to this question is positive.
	An equilateral set of three points in a Banach space always lies on a sphere, but this is not the case anymore for an equilateral set of size 4 \cite{BCP93}.
	We are not aware of an answer to this question in the tropical setting.
\end{remark}
\vspace{5mm}
\noindent
{\bf Acknowledgments.} We thank K. Swanepoel for bringing to our attention that the results about the equilateral dimension already appeared in his works.

\bibliographystyle{plain}
\bibliography{Tropical_chromatic_number}

\begin{thebibliography}{10}

\bibitem{AM83}
Noga Alon and Vitali~D. Milman.
\newblock Embedding of {$l\sp{k}\sb{\infty }$}\ in finite-dimensional {B}anach
  spaces.
\newblock {\em Israel J. Math.}, 45(4):265--280, 1983.

\bibitem{AP03}
Noga Alon and Pavel Pudl\'ak.
\newblock Equilateral sets in {$l^n_p$}.
\newblock {\em Geom. Funct. Anal.}, 13(3):467--482, 2003.

\bibitem{BBMP19}
Christine Bachoc, Thomas Bellitto, Philippe Moustrou, and Arnaud P\^echer.
\newblock On the density of sets avoiding parallelohedron distance 1.
\newblock {\em Discrete Comput. Geom.}, 62(3):497--524, 2019.

\bibitem{BCP93}
M.~Baronti, E.~Casini, and P.~L. Papini.
\newblock Equilateral sets and their central points.
\newblock {\em Rend. Mat. Appl. (7)}, 13(1):133--148, 1993.

\bibitem{Cadical24}
Armin Biere, Tobias Faller, Katalin Fazekas, Mathias Fleury, Nils Froleyks, and
  Florian Pollitt.
\newblock {CaDiCaL 2.0}.
\newblock In Arie Gurfinkel and Vijay Ganesh, editors, {\em Computer Aided
  Verification - 36th International Conference, {CAV} 2024, Montreal, QC,
  Canada, July 24-27, 2024, Proceedings, Part {I}}, volume 14681 of {\em
  Lecture Notes in Computer Science}, pages 133--152. Springer, 2024.

\bibitem{Bol65}
B\'ela Bollob\'as.
\newblock On generalized graphs.
\newblock {\em Acta Math. Acad. Sci. Hungar.}, 16:447--452, 1965.

\bibitem{But10}
Peter Butkovi{\v{c}}.
\newblock {\em Max-linear systems: theory and algorithms}.
\newblock Springer Monographs in Mathematics. Springer-Verlag London, Ltd.,
  London, 2010.

\bibitem{CKR18}
Danila Cherkashin, Anatoly Kulikov, and Andrei Raigorodskii.
\newblock On the chromatic numbers of small-dimensional {E}uclidean spaces.
\newblock {\em Discrete Appl. Math.}, 243:125--131, 2018.

\bibitem{CGQ04}
Guy Cohen, St\'ephane Gaubert, and Jean-Pierre Quadrat.
\newblock Duality and separation theorems in idempotent semimodules.
\newblock {\em Linear Algebra Appl.}, 379:395--422, 2004.

\bibitem{CJS22}
Francisco Criado, Michael Joswig, and Francisco Santos.
\newblock Tropical bisectors and {V}oronoi diagrams.
\newblock {\em Found. Comput. Math.}, 22(6):1923--1960, 2022.

\bibitem{DG62}
L.~Danzer and B.~Gr\"unbaum.
\newblock \"uber zwei {P}robleme bez\"uglich konvexer {K}\"orper von {P}.
  {E}rdős und von {V}. {L}. {K}lee.
\newblock {\em Math. Z.}, 79:95--99, 1962.

\bibitem{deBE51}
Nicolass~G. de~Bruijn and Paul Erd\"os.
\newblock A colour problem for infinite graphs and a problem in the theory of
  relations.
\newblock {\em Indag. Math.}, 13:369--373, 1951.
\newblock Nederl. Akad. Wetensch. Proc. Ser. A {\bf 54}.

\bibitem{deG18}
Aubrey D. N.~J. de~Grey.
\newblock The chromatic number of the plane is at least 5.
\newblock {\em Geombinatorics}, 28(1):18--31, 2018.

\bibitem{Pue14}
Mar\'ia~Jes\'us de~la Puente.
\newblock Distances on the tropical line determined by two points.
\newblock {\em Kybernetika (Prague)}, 50(3):408--435, 2014.

\bibitem{FLM91}
Zolt\'an F\"uredi, Jeffrey~C. Lagarias, and Frank Morgan.
\newblock Singularities of minimal surfaces and networks and related extremal
  problems in {M}inkowski space.
\newblock In {\em Discrete and computational geometry ({N}ew {B}runswick, {NJ},
  1989/1990)}, volume~6 of {\em DIMACS Ser. Discrete Math. Theoret. Comput.
  Sci.}, pages 95--109. Amer. Math. Soc., Providence, RI, 1991.

\bibitem{Guy83}
Richard~K. Guy.
\newblock Unsolved {P}roblems: {A}n {O}lla-{P}odrida of {O}pen {P}roblems,
  {O}ften {O}ddly {P}osed.
\newblock {\em Amer. Math. Monthly}, 90(3):196--200, 1983.

\bibitem{HH18}
Q.~Huangfu and J.~A.~J. Hall.
\newblock Parallelizing the dual revised simplex method.
\newblock {\em Mathematical Programming Computation}, 10(1):119--142, 2018.

\bibitem{Kup11}
Andrey Kupavskiy.
\newblock On the chromatic number of {$\Bbb R^n$} with an arbitrary norm.
\newblock {\em Discrete Math.}, 311(6):437--440, 2011.

\bibitem{LR72}
D.~G. Larman and C.~A. Rogers.
\newblock The realization of distances within sets in {E}uclidean space.
\newblock {\em Mathematika}, 19:1--24, 1972.

\bibitem{MS15}
Diane Maclagan and Bernd Sturmfels.
\newblock {\em Introduction to tropical geometry}, volume 161 of {\em Graduate
  Studies in Mathematics}.
\newblock American Mathematical Society, Providence, RI, 2015.

\bibitem{MSW01}
Horst Martini, Konrad~J. Swanepoel, and Gunter Wei\ss.
\newblock The geometry of {M}inkowski spaces---a survey. {I}.
\newblock {\em Expo. Math.}, 19(2):97--142, 2001.

\bibitem{MM61}
Leo Moser and William Moser.
\newblock Solution to problem 10.
\newblock {\em Can. Math. Bull}, 4:187--189, 1961.

\bibitem{Paarts20}
Jaan Parts.
\newblock Graph minimization, focusing on the example of 5-chromatic
  unit-distance graphs in the plane.
\newblock {\em Geombinatorics}, 29(4):137--166, 2020.

\bibitem{Petty71}
Clinton~M. Petty.
\newblock Equilateral sets in {M}inkowski spaces.
\newblock {\em Proc. Amer. Math. Soc.}, 29:369--374, 1971.

\bibitem{Ra00}
Andrei~M. Raigorodskii.
\newblock On the chromatic number of a space.
\newblock {\em Uspekhi Mat. Nauk}, 55(2(332)):147--148, 2000.

\bibitem{Ros26}
Amnon Rosenmann.
\newblock Tropical balls, geodesics and honeycomb.
\newblock 2026.
\newblock arXiv:2601.14447.

\bibitem{SchS06}
Achill Sch\"urmann and Konrad~J. Swanepoel.
\newblock Three-dimensional antipodal and norm-equilateral sets.
\newblock {\em Pacific J. Math.}, 228(2):349--370, 2006.

\bibitem{Soi09}
Alexander Soifer.
\newblock {\em The mathematical coloring book}.
\newblock Springer, New York, 2009.

\bibitem{Soltan75}
Petru~S. Soltan.
\newblock Analogues of regular simplexes in normed spaces.
\newblock {\em Dokl. Akad. Nauk SSSR}, 222(6):1303--1305, 1975.

\bibitem{Sperner28}
Emanuel Sperner.
\newblock Ein {S}atz \"uber {U}ntermengen einer endlichen {M}enge.
\newblock {\em Math. Z.}, 27(1):544--548, 1928.

\bibitem{Swan04}
Konrad~J. Swanepoel.
\newblock Equilateral sets in finite-dimensional normed spaces.
\newblock In {\em Seminar of {M}athematical {A}nalysis}, volume~71 of {\em
  Colecc. Abierta}, pages 195--237. Univ. Sevilla Secr. Publ., Seville, 2004.

\bibitem{Swan07}
Konrad~J. Swanepoel.
\newblock The local {S}teiner problem in finite-dimensional normed spaces.
\newblock {\em Discrete Comput. Geom.}, 37(3):419--442, 2007.

\bibitem{Swan18}
Konrad~J. Swanepoel.
\newblock Combinatorial distance geometry in normed spaces.
\newblock In {\em New trends in intuitive geometry}, volume~27 of {\em Bolyai
  Soc. Math. Stud.}, pages 407--458. J\'anos Bolyai Math. Soc., Budapest, 2018.

\bibitem{SV08}
Konrad~J. Swanepoel and Rafael Villa.
\newblock A lower bound for the equilateral number of normed spaces.
\newblock {\em Proc. Amer. Math. Soc.}, 136(1):127--131, 2008.

\end{thebibliography}

\appendix
\section{A graph in $\mathbb{Z}^4$ that is not 10-colorable}
\label{sec:Appendix A}
The graph contains 37 vertices and 386 edges between vertices of tropical distance 2.
\begin{longtable}{|c|c|c|}
	\hline
	Index & Vertex & Color \\
	\hline
	\endfirsthead
	
	\hline
	\endlastfoot
	
	1 & (1,1,0,0) & 0 \\
	2 & (1,0,1,0) & 1 \\
	3 & (1,0,0,1) & 2 \\
	4 & (0,-1,-1,-1) & 3 \\
	5 & (0,1,1,0) & 4 \\
	6 & (0,1,0,1) & 5 \\
	7 & (-1,0,-1,-1) & 6 \\
	8 & (0,0,1,1) & 7 \\
	9 & (-1,-1,0,-1) & 8 \\
	10 & (-1,-1,-1,0) & 9 \\
	11 & (2,2,2,1) & 6 \\
	12 & (2,2,1,2) & 8 \\
	13 & (2,1,2,2) & 3 \\
	14 & (1,2,2,2) & 9 \\
	15 & (1,0,0,0) & 0 \\
	16 & (0,1,0,0) & 4 \\
	17 & (0,0,1,0) & 1 \\
	18 & (0,0,0,1) & 5 \\
	19 & (-1,-1,-1,-1) & 3 \\
	20 & (2,2,1,1) & 10 \\
	21 & (2,1,2,1) & 3 \\
	22 & (2,1,1,2) & 8 \\
	23 & (1,2,2,1) & 6 \\
	24 & (1,2,1,2) & 9 \\
	25 & (1,1,2,2) & 7 \\
	26 & (1,1,0,-1) & 2 \\
	27 & (1,1,-1,0) & 1 \\
	28 & (1,0,1,-1) & 7 \\
	29 & (1,0,-1,1) & 2 \\
	30 & (1,-1,1,0) & 4 \\
	31 & (1,-1,0,1) & 6 \\
	32 & (0,1,1,-1) & 5 \\
	33 & (0,1,-1,1) & 7 \\
	34 & (0,-1,1,1) & 10 \\
	35 & (-1,1,1,0) & 2 \\
	36 & (-1,1,0,1) & 10 \\
	37 & (-1,0,1,1) & 0 \\
\end{longtable}

{\small
	\setlength{\tabcolsep}{3pt}
	\begin{longtable}{|c|c|c|c|c|c|c|c|}
		\hline
		\multicolumn{8}{|c|}{Edges} \\
		\hline
		\endfirsthead
		
		\hline
		\endlastfoot
		
		(1,2) & (1,3) & (1,4) & (1,5) & (1,6) & (1,7) & (1,8) & (1,9) \\
		(1,10) & (1,11) & (1,12) & (1,13) & (1,14) & (1,17) & (1,18) & (1,19) \\
		(1,21) & (1,22) & (1,23) & (1,24) & (1,25) & (1,28) & (1,29) & (1,32) \\
		(1,33) & (2,3) & (2,4) & (2,5) & (2,6) & (2,7) & (2,8) & (2,9) \\
		(2,10) & (2,11) & (2,12) & (2,13) & (2,14) & (2,16) & (2,18) & (2,19) \\
		(2,20) & (2,22) & (2,23) & (2,24) & (2,25) & (2,26) & (2,31) & (2,32) \\
		(2,34) & (3,4) & (3,5) & (3,6) & (3,7) & (3,8) & (3,9) & (3,10) \\
		(3,11) & (3,12) & (3,13) & (3,14) & (3,16) & (3,17) & (3,19) & (3,20) \\
		(3,21) & (3,23) & (3,24) & (3,25) & (3,27) & (3,30) & (3,33) & (3,34) \\
		(4,5) & (4,6) & (4,7) & (4,8) & (4,9) & (4,10) & (4,16) & (4,17) \\
		(4,18) & (4,26) & (4,27) & (4,28) & (4,29) & (4,30) & (4,31) & (4,32) \\
		(4,33) & (4,34) & (5,6) & (5,7) & (5,8) & (5,9) & (5,10) & (5,11) \\
		(5,12) & (5,13) & (5,14) & (5,15) & (5,18) & (5,19) & (5,20) & (5,21) \\
		(5,22) & (5,24) & (5,25) & (5,26) & (5,28) & (5,36) & (5,37) & (6,7) \\
		(6,8) & (6,9) & (6,10) & (6,11) & (6,12) & (6,13) & (6,14) & (6,15) \\
		(6,17) & (6,19) & (6,20) & (6,21) & (6,22) & (6,23) & (6,25) & (6,27) \\
		(6,29) & (6,35) & (6,37) & (7,8) & (7,9) & (7,10) & (7,15) & (7,17) \\
		(7,18) & (7,26) & (7,27) & (7,28) & (7,29) & (7,32) & (7,33) & (7,35) \\
		(7,36) & (7,37) & (8,9) & (8,10) & (8,11) & (8,12) & (8,13) & (8,14) \\
		(8,15) & (8,16) & (8,19) & (8,20) & (8,21) & (8,22) & (8,23) & (8,24) \\
		(8,30) & (8,31) & (8,35) & (8,36) & (9,10) & (9,15) & (9,16) & (9,18) \\
		(9,26) & (9,28) & (9,30) & (9,31) & (9,32) & (9,34) & (9,35) & (9,36) \\
		(9,37) & (10,15) & (10,16) & (10,17) & (10,27) & (10,29) & (10,30) & (10,31) \\
		(10,33) & (10,34) & (10,35) & (10,36) & (10,37) & (11,12) & (11,13) & (11,14) \\
		(11,15) & (11,16) & (11,17) & (11,18) & (11,22) & (11,24) & (11,25) & (11,26) \\
		(11,28) & (11,32) & (12,13) & (12,14) & (12,15) & (12,16) & (12,17) & (12,18) \\
		(12,21) & (12,23) & (12,25) & (12,27) & (12,29) & (12,33) & (13,14) & (13,15) \\
		(13,16) & (13,17) & (13,18) & (13,20) & (13,23) & (13,24) & (13,30) & (13,31) \\
		(13,34) & (14,15) & (14,16) & (14,17) & (14,18) & (14,20) & (14,21) & (14,22) \\
		(14,35) & (14,36) & (14,37) & (15,16) & (15,17) & (15,18) & (15,19) & (15,20) \\
		(15,21) & (15,22) & (15,23) & (15,24) & (15,25) & (15,26) & (15,27) & (15,28) \\
		(15,29) & (15,30) & (15,31) & (15,32) & (15,33) & (15,34) & (16,17) & (16,18) \\
		(16,19) & (16,20) & (16,21) & (16,22) & (16,23) & (16,24) & (16,25) & (16,26) \\
		(16,27) & (16,28) & (16,29) & (16,32) & (16,33) & (16,35) & (16,36) & (16,37) \\
		(17,18) & (17,19) & (17,20) & (17,21) & (17,22) & (17,23) & (17,24) & (17,25) \\
		(17,26) & (17,28) & (17,30) & (17,31) & (17,32) & (17,34) & (17,35) & (17,36) \\
		(17,37) & (18,19) & (18,20) & (18,21) & (18,22) & (18,23) & (18,24) & (18,25) \\
		(18,27) & (18,29) & (18,30) & (18,31) & (18,33) & (18,34) & (18,35) & (18,36) \\
		(18,37) & (19,26) & (19,27) & (19,28) & (19,29) & (19,30) & (19,31) & (19,32) \\
		(19,33) & (19,34) & (19,35) & (19,36) & (19,37) & (20,21) & (20,22) & (20,23) \\
		(20,24) & (20,25) & (20,26) & (20,27) & (20,28) & (20,29) & (20,32) & (20,33) \\
		(21,22) & (21,23) & (21,24) & (21,25) & (21,26) & (21,28) & (21,30) & (21,31) \\
		(21,32) & (21,34) & (22,23) & (22,24) & (22,25) & (22,27) & (22,29) & (22,30) \\
		(22,31) & (22,33) & (22,34) & (23,24) & (23,25) & (23,26) & (23,28) & (23,32) \\
		(23,35) & (23,36) & (23,37) & (24,25) & (24,27) & (24,29) & (24,33) & (24,35) \\
		(24,36) & (24,37) & (25,30) & (25,31) & (25,34) & (25,35) & (25,36) & (25,37) \\
		(26,27) & (26,28) & (26,32) & (27,29) & (27,33) & (28,30) & (28,32) & (29,31) \\
		(29,33) & (30,31) & (30,34) & (31,34) & (32,35) & (33,36) & (34,37) & (35,36) \\
		(35,37) & (36,37) &  &  &  &  &  &
		\label{tab:edges-11-color-dim4}
	\end{longtable}
}
\vspace{4mm}
CaDiCaL was very slow when trying to prove that the graph is not 10-colorable. However, when added the fact that the subset of vertices $\{1,2,3,5,6,8,11,12,13,14\}$ forms a clique (they are of tropical distance 2 from one another) then the result was UNSATISFIABLE after 0.01 seconds.
\vspace{4mm}

CaDiCaL output on 11-coloring:
\begin{verbatim}
c parsing ...
c done (0.327 s)
c solving ...
c done (0.176 s)
s SATISFIABLE
\end{verbatim}
\vspace{4mm}

\section{A graph in $\mathbb{Z}^3$ that is not 7-colorable}
\label{sec:Appendix B}
The graph contains 62 vertices and 577 edges between vertices of tropical distance 2.
\begin{longtable}{|c|c|c|}
	\hline
	Index & Vertex & Color \\
	\hline
	\endfirsthead
	\hline
	\endlastfoot
	1 & $(0,1,0)$ & 0\\
	2 & $(0,1,1)$ & 0\\
	3 & $(0,1,2)$ & 6\\
	4 & $(0,2,0)$ & 6\\
	5 & $(0,2,1)$ & 1\\
	6 & $(0,3,1)$ & 5\\
	7 & $(0,3,2)$ & 1\\
	8 & $(1,0,0)$ & 7\\
	9 & $(1,0,1)$ & 2\\
	10 & $(1,1,0)$ & 7\\
	11 & $(1,1,1)$ & 0\\
	12 & $(1,1,2)$ & 4\\
	13 & $(1,1,3)$ & 6\\
	14 & $(1,2,0)$ & 4\\
	15 & $(1,2,1)$ & 0\\
	16 & $(1,2,2)$ & 3\\
	17 & $(1,2,3)$ & 6\\
	18 & $(1,3,1)$ & 6\\
	19 & $(1,3,2)$ & 3\\
	20 & $(1,3,3)$ & 7\\
	21 & $(1,3,4)$ & 5\\
	22 & $(1,4,1)$ & 5\\
	23 & $(1,4,2)$ & 5\\
	24 & $(1,4,3)$ & 7\\
	25 & $(2,0,0)$ & 1\\
	26 & $(2,0,2)$ & 3\\
	27 & $(2,1,2)$ & 4\\
	28 & $(2,1,3)$ & 4\\
	29 & $(2,2,0)$ & 3\\
	30 & $(2,2,1)$ & 5\\
	31 & $(2,2,2)$ & 5\\
	32 & $(2,2,3)$ & 4\\
	33 & $(2,2,4)$ & 2\\
	34 & $(2,3,1)$ & 4\\
	35 & $(2,3,2)$ & 2\\
	36 & $(2,3,3)$ & 3\\
	37 & $(2,4,2)$ & 2\\
	38 & $(2,4,3)$ & 2\\
	39 & $(2,4,4)$ & 7\\
	40 & $(3,1,1)$ & 6\\
	41 & $(3,1,2)$ & 6\\
	42 & $(3,1,3)$ & 3\\
	43 & $(3,2,0)$ & 3\\
	44 & $(3,2,1)$ & 3\\
	45 & $(3,2,2)$ & 5\\
	46 & $(3,2,4)$ & 7\\
	47 & $(3,3,1)$ & 3\\
	48 & $(3,3,3)$ & 1\\
	49 & $(3,3,4)$ & 0\\
	50 & $(3,4,1)$ & 7\\
	51 & $(3,4,2)$ & 7\\
	52 & $(3,4,3)$ & 2\\
	53 & $(3,4,4)$ & 0\\
	54 & $(4,2,2)$ & 2\\
	55 & $(4,2,3)$ & 2\\
	56 & $(4,2,4)$ & 7\\
	57 & $(4,3,2)$ & 0\\
	58 & $(4,3,3)$ & 1\\
	59 & $(4,3,4)$ & 7\\
	60 & $(4,4,2)$ & 0\\
	61 & $(4,4,3)$ & 1\\
	62 & $(4,4,4)$ & 6
	\label{tab:62_vertices}
\end{longtable}


{\small
	\setlength{\tabcolsep}{3pt}
	\begin{longtable}{|c|c|c|c|c|c|c|c|}
		\hline
		\multicolumn{8}{|c|}{Edges} \\
		\hline
		\endfirsthead
		\hline
		\endlastfoot
		(1,3) & (1,6) & (1,7) & (1,8) & (1,9) & (1,12) & (1,16) & (1,18) \\
		(1,19) & (1,27) & (1,29) & (1,30) & (1,31) & (1,34) & (1,35) & (2,4) \\
		(2,6) & (2,7) & (2,8) & (2,9) & (2,10) & (2,13) & (2,14) & (2,17) \\
		(2,18) & (2,19) & (2,20) & (2,27) & (2,28) & (2,30) & (2,31) & (2,32) \\
		(2,34) & (2,35) & (2,36) & (3,5) & (3,7) & (3,9) & (3,11) & (3,15) \\
		(3,19) & (3,20) & (3,21) & (3,27) & (3,28) & (3,31) & (3,32) & (3,33) \\
		(3,35) & (3,36) & (4,7) & (4,10) & (4,11) & (4,16) & (4,19) & (4,22) \\
		(4,23) & (4,29) & (4,30) & (4,31) & (4,34) & (4,35) & (4,37) & (5,10) \\
		(5,11) & (5,12) & (5,14) & (5,17) & (5,20) & (5,22) & (5,23) & (5,24) \\
		(5,30) & (5,31) & (5,32) & (5,34) & (5,35) & (5,36) & (5,37) & (5,38) \\
		(6,14) & (6,15) & (6,16) & (6,20) & (6,24) & (6,34) & (6,35) & (6,36) \\
		(6,37) & (6,38) & (7,15) & (7,16) & (7,17) & (7,18) & (7,21) & (7,22) \\
		(7,35) & (7,36) & (7,37) & (7,38) & (7,39) & (8,12) & (8,14) & (8,15) \\
		(8,16) & (8,26) & (8,27) & (8,29) & (8,30) & (8,31) & (8,40) & (8,41) \\
		(8,43) & (8,44) & (8,45) & (9,10) & (9,13) & (9,15) & (9,16) & (9,17) \\
		(9,25) & (9,28) & (9,30) & (9,31) & (9,32) & (9,40) & (9,41) & (9,42) \\
		(9,44) & (9,45) & (10,12) & (10,16) & (10,18) & (10,19) & (10,25) & (10,27) \\
		(10,31) & (10,34) & (10,35) & (10,40) & (10,41) & (10,43) & (10,44) & (10,45) \\
		(10,47) & (11,13) & (11,14) & (11,17) & (11,18) & (11,19) & (11,20) & (11,25) \\
		(11,26) & (11,28) & (11,29) & (11,32) & (11,34) & (11,35) & (11,36) & (11,40) \\
		(11,41) & (11,42) & (11,44) & (11,45) & (11,47) & (11,48) & (12,15) & (12,19) \\
		(12,20) & (12,21) & (12,26) & (12,30) & (12,33) & (12,35) & (12,36) & (12,41) \\
		(12,42) & (12,45) & (12,46) & (12,48) & (12,49) & (13,16) & (13,20) & (13,21) \\
		(13,26) & (13,27) & (13,31) & (13,36) & (13,42) & (13,46) & (13,48) & (13,49) \\
		(14,16) & (14,19) & (14,22) & (14,23) & (14,31) & (14,35) & (14,37) & (14,43) \\
		(14,44) & (14,45) & (14,47) & (14,50) & (14,51) & (15,17) & (15,20) & (15,22) \\
		(15,23) & (15,24) & (15,27) & (15,29) & (15,32) & (15,36) & (15,37) & (15,38) \\
		(15,44) & (15,45) & (15,47) & (15,48) & (15,50) & (15,51) & (15,52) & (16,18) \\
		(16,21) & (16,23) & (16,24) & (16,27) & (16,28) & (16,30) & (16,33) & (16,34) \\
		(16,37) & (16,38) & (16,39) & (16,45) & (16,46) & (16,48) & (16,49) & (16,51) \\
		(16,52) & (16,53) & (17,19) & (17,24) & (17,27) & (17,28) & (17,31) & (17,35) \\
		(17,38) & (17,39) & (17,46) & (17,48) & (17,49) & (17,52) & (17,53) & (18,20) \\
		(18,24) & (18,29) & (18,30) & (18,31) & (18,36) & (18,38) & (18,47) & (18,48) \\
		(18,50) & (18,51) & (18,52) & (19,21) & (19,22) & (19,30) & (19,31) & (19,32) \\
		(19,34) & (19,39) & (19,48) & (19,49) & (19,51) & (19,52) & (19,53) & (20,23) \\
		(20,31) & (20,32) & (20,33) & (20,35) & (20,37) & (20,48) & (20,49) & (20,52) \\
		(20,53) & (21,24) & (21,32) & (21,33) & (21,36) & (21,38) & (21,49) & (21,53) \\
		(22,24) & (22,34) & (22,35) & (22,38) & (22,50) & (22,51) & (22,52) & (23,34) \\
		(23,35) & (23,36) & (23,39) & (23,51) & (23,52) & (23,53) & (24,35) & (24,36) \\
		(24,37) & (24,52) & (24,53) & (25,26) & (25,27) & (25,29) & (25,30) & (25,31) \\
		(25,41) & (25,43) & (25,44) & (25,45) & (25,54) & (26,31) & (26,32) & (26,33) \\
		(26,40) & (26,45) & (26,46) & (26,54) & (26,55) & (26,56) & (27,30) & (27,33) \\
		(27,35) & (27,36) & (27,40) & (27,44) & (27,46) & (27,48) & (27,49) & (27,54) \\
		(27,55) & (27,56) & (27,57) & (27,58) & (27,59) & (28,31) & (28,36) & (28,41) \\
		(28,45) & (28,48) & (28,49) & (28,55) & (28,56) & (28,58) & (28,59) & (29,31) \\
		(29,35) & (29,37) & (29,40) & (29,45) & (29,50) & (29,51) & (29,54) & (29,57) \\
		(29,60) & (30,32) & (30,36) & (30,37) & (30,38) & (30,40) & (30,41) & (30,43) \\
		(30,48) & (30,50) & (30,51) & (30,52) & (30,54) & (30,55) & (30,57) & (30,58) \\
		(30,60) & (30,61) & (31,33) & (31,34) & (31,37) & (31,38) & (31,39) & (31,40) \\
		(31,41) & (31,42) & (31,44) & (31,46) & (31,47) & (31,49) & (31,51) & (31,52) \\
		(31,53) & (31,54) & (31,55) & (31,56) & (31,57) & (31,58) & (31,59) & (31,60) \\
		(31,61) & (31,62) & (32,35) & (32,38) & (32,39) & (32,41) & (32,42) & (32,45) \\
		(32,52) & (32,53) & (32,55) & (32,56) & (32,58) & (32,59) & (32,61) & (32,62) \\
		(33,36) & (33,39) & (33,42) & (33,48) & (33,53) & (33,56) & (33,59) & (33,62) \\
		(34,36) & (34,38) & (34,43) & (34,44) & (34,45) & (34,48) & (34,52) & (34,57) \\
		(34,58) & (34,60) & (34,61) & (35,39) & (35,44) & (35,45) & (35,47) & (35,49) \\
		(35,50) & (35,53) & (35,57) & (35,58) & (35,59) & (35,60) & (35,61) & (35,62) \\
		(36,37) & (36,45) & (36,46) & (36,51) & (36,58) & (36,59) & (36,61) & (36,62) \\
		(37,39) & (37,47) & (37,48) & (37,50) & (37,53) & (37,60) & (37,61) & (37,62) \\
		(38,48) & (38,49) & (38,51) & (38,61) & (38,62) & (39,48) & (39,49) & (39,52) \\
		(39,62) & (40,42) & (40,43) & (40,47) & (40,48) & (40,55) & (40,57) & (40,58) \\
		(41,44) & (41,46) & (41,48) & (41,49) & (41,56) & (41,57) & (41,58) & (41,59) \\
		(42,45) & (42,48) & (42,49) & (42,54) & (42,58) & (42,59) & (43,45) & (43,50) \\
		(43,51) & (43,54) & (43,57) & (43,60) & (44,48) & (44,50) & (44,51) & (44,52) \\
		(44,55) & (44,58) & (44,60) & (44,61) & (45,46) & (45,47) & (45,49) & (45,51) \\
		(45,52) & (45,53) & (45,56) & (45,59) & (45,60) & (45,61) & (45,62) & (46,48) \\
		(46,53) & (46,55) & (46,58) & (46,62) & (47,48) & (47,52) & (47,54) & (47,58) \\
		(47,61) & (48,51) & (48,54) & (48,55) & (48,56) & (48,57) & (48,60) & (49,52) \\
		(49,55) & (49,56) & (49,58) & (49,61) & (50,52) & (50,57) & (50,61) & (51,53) \\
		(51,57) & (51,58) & (51,62) & (52,57) & (52,58) & (52,59) & (52,60) & (53,58) \\
		(53,59) & (53,61) & (54,56) & (54,59) & (54,60) & (54,61) & (54,62) & (55,57) \\
		(55,61) & (55,62) & (56,58) & (56,62) & (57,59) & (57,62) & (58,60) & (59,61) \\
		(60,62) &  &  &  &  &  &  & 
	\end{longtable}
}
\vspace{4mm}
CaDiCaL output on 7-coloring:
\begin{verbatim}
c parsing ...
c done (0.294 s)
c solving ...
c done (1094.426 s)
s UNSATISFIABLE
\end{verbatim}
\vspace{4mm}

CaDiCaL output on 8-coloring:
\begin{verbatim}
c parsing ...
c done (0.446 s)
c solving ...
c done (0.529 s)
s SATISFIABLE
\end{verbatim}
\vspace{4mm}
\end{document}